\documentclass[10pt,a4paper]{article}

\usepackage{latexsym,amsfonts,amsmath,amssymb,mathrsfs,url,color,amsthm}
\usepackage{graphicx}

\newtheorem{theorem}{Theorem}
\newtheorem{lemma}{Lemma}
\newtheorem{algorithm}{Algorithm}
\newtheorem{remark}{Remark}

\newtheorem{definition}{Definition}
\newtheorem{corollary}{Corollary}

\usepackage[dvipdfm,
linkbordercolor={1 0 0},
citebordercolor={0 1 0},
urlbordercolor={0 1 1}]{hyperref}

\newcommand{\rd}{\,\mathrm{d}}
\newcommand{\rtr}{\,\mathrm{tr}}
\newcommand{\bsa}{\boldsymbol{a}}
\newcommand{\bsb}{\boldsymbol{b}}

\newcommand{\bsk}{\boldsymbol{k}}
\newcommand{\bsl}{\boldsymbol{l}}
\newcommand{\bsq}{\boldsymbol{q}}

\newcommand{\bsx}{\boldsymbol{x}}
\newcommand{\bsy}{\boldsymbol{y}}
\newcommand{\bsz}{\boldsymbol{z}}

\newcommand{\bsP}{\boldsymbol{P}}
\newcommand{\bsQ}{\boldsymbol{Q}}

\newcommand{\bsU}{\boldsymbol{U}}
\newcommand{\bsV}{\boldsymbol{V}}
\newcommand{\bsW}{\boldsymbol{W}}
\newcommand{\bsalpha}{\boldsymbol{\alpha}}

\newcommand{\bsgamma}{\boldsymbol{\gamma}}
\newcommand{\bssigma}{\boldsymbol{\sigma}}
\newcommand{\bstau}{\boldsymbol{\tau}}

\newcommand{\bszero}{\boldsymbol{0}}

\newcommand{\nat}{\mathbb{N}}

\newcommand{\FF}{\mathbb{F}}
\newcommand{\RR}{\mathbb{R}}

\newcommand{\Dcal}{\mathcal{D}}
\newcommand{\Ecal}{\mathcal{E}}

\newcommand{\Hcal}{\mathcal{H}}

\newcommand{\Kcal}{\mathcal{K}}
\newcommand{\Lcal}{\mathcal{L}}

\newcommand{\wal}{\mathrm{wal}}

\allowdisplaybreaks

\begin{document}

\title{Constructing good higher order polynomial lattice rules with modulus of reduced degree}

\author{Takashi Goda\thanks{Graduate School of Engineering, The University of Tokyo, 7-3-1 Hongo, Bunkyo-ku, Tokyo 113-8656 ({\tt goda@frcer.t.u-tokyo.ac.jp}).}}

\date{\today}

\maketitle

\begin{abstract}
In this paper we investigate multivariate integration in weighted unanchored Sobolev spaces of smoothness of arbitrarily high order. As quadrature points we employ higher order polynomial lattice point sets over $\FF_2$ which are randomly digitally shifted and then folded using the tent transformation. We first prove the existence of good higher order polynomial lattice rules which achieve the optimal rate of the mean square worst-case error, while reducing the required degree of modulus by half as compared to higher order polynomial lattice rules whose quadrature points are randomly digitally shifted but not folded using the tent transformation. Thus we are able to restrict the search space of generating vectors significantly. We then study the component-by-component construction as an explicit means of obtaining good higher order polynomial lattice rules. In a way analogous to [J. Baldeaux, J. Dick, G. Leobacher, D. Nuyens, F. Pillichshammer, Numer. Algorithms, 59 (2012) 403--431], we show how to calculate the quality criterion efficiently and how to obtain the fast component-by-component construction using the fast Fourier transform. Our result generalizes the previous result shown by [L.~L. Cristea, J. Dick, G. Leobacher, F. Pillichshammer, Numer. Math., 105 (2007) 413--455], in which the degree of smoothness is fixed at 2 and classical polynomial lattice rules are considered.
\end{abstract}

\section{Introduction}\label{sec:intro}
In this paper we study multivariate integration of smooth functions defined over the $s$-dimensional unit cube $[0,1)^s$,
  \begin{align*}
     I(f)=\int_{[0,1)^s}f(\bsx)\rd\bsx .
  \end{align*}
Quasi-Monte Carlo (QMC) rules approximate $I(f)$ by
  \begin{align*}
     Q(f;P_N)=\frac{1}{N}\sum_{n=0}^{N-1}f(\bsx_n) ,
  \end{align*}
where a point set $P_{N}=\{\bsx_0,\ldots ,\bsx_{N-1}\}\subset [0,1)^s$ is chosen carefully so as to yield a small integration error.

Explicit constructions of point sets whose star-discrepancy is of order $N^{-1+\epsilon}$ for any $\epsilon>0$ have been studied extensively. They are motivated by the so-called Koksma-Hlawka inequality, which states that the integration error $|I(f)-Q(f;P_N)|$ is bounded above by the variation of $f$ in the sense of Hardy and Krause times the star-discrepancy of $P_{N}$. There are two prominent families for construction of good point sets: integration lattices \cite{Ni92b,SJ94} and digital nets and sequences \cite{DP10,Ni92b}. Regarding explicit constructions of digital sequences, we refer to \cite[Chapter~8]{DP10} and \cite[Chapter~4]{Ni92b}.

Polynomial lattice point sets, first proposed in \cite{Ni92a}, are a well-known special construction of digital nets based on rational functions over finite fields. The existence of low-discrepancy polynomial lattice point sets has been proven previously, see, e.g., \cite{KP12,La93}. QMC rules using polynomial lattice point sets as $P_{N}$ are called polynomial lattice rules, which are defined as follows. We refer to \cite{DP10,Nuxx} for more information on polynomial lattice rules.

For a prime $b$, let $\FF_b:=\{0,\ldots,b-1\}$ be the finite field consisting of $b$ elements. We denote by $\FF_b[x]$ the set of all polynomials over $\FF_b$ and by $\FF_b((x^{-1}))$ the field of formal Laurent series over $\FF_b$. Every element of $\FF_b((x^{-1}))$ can be represented as
  \begin{align*}
    L=\sum_{l=w}^{\infty}t_lx^{-l},
  \end{align*}
for some integer $w$ and $t_l\in \FF_b$. For a positive integer $m$, we define the mapping $v_m$ from $\FF_b((x^{-1}))$ to the unit interval $[0,1)$ by
  \begin{align*}
    v_m\left( \sum_{l=w}^{\infty}t_l x^{-l}\right) =\sum_{l=\max(1,w)}^{m}t_l b^{-l}.
  \end{align*}
We often identify an integer $n=n_0+n_1b+\cdots\in \nat_0$ with a polynomial $n(x)=n_0+n_1x+\cdots \in \FF_b[x]$, where we denote $\nat_0=\nat \cup \{0\}$ with $\nat$ the set of positive integers. Using these notations, polynomial lattice rules are defined as follows.

\begin{definition}\label{def:polynomial_lattice}
For $m, s \in \nat$, let $p \in \FF_b[x]$ with $\deg(p)=m$ and let $\bsq=(q_1,\ldots,q_s) \in (\FF_b[x])^s$. A polynomial lattice rule over $\FF_b$ is a QMC rule using a polynomial lattice point set $P_{b^m}(\bsq,p)$ consisting of $b^m$ points that are given by
  \begin{align*}
    \bsx_n &:= \left( v_m\left( \frac{n(x)q_1(x)}{p(x)} \right) , \ldots , v_m\left( \frac{n(x)q_s(x)}{p(x)} \right) \right) \in [0,1)^s ,
  \end{align*}
for $0\le n <b^m$. The vector $\bsq$ and the polynomial $p$ are respectively called the {\em generating vector} and the {\em modulus} of $P_{b^m}(\bsq,p)$.
\end{definition}

The aim of this paper is to construct good point sets which can exploit the smoothness of an integrand so as to improve the convergence rate of the integration error. Two principles for constructing such point sets based on the concept of digital nets have been proposed so far. One is known as {\em higher order polynomial lattice rules} that are given by generalizing the definition of polynomial lattice rules, see, e.g., \cite{BDGP11,BDLNP12,DP07}. The other is based on a digit interlacing function applied to digital nets and sequences whose number of components is a multiple of the dimension, see \cite{Di07,Di08}. Since we focus on the former principle in this paper, we only give the definition of higher order polynomial lattice rules in the following. 

\begin{definition}\label{def:ho_polynomial_lattice}
For $m, m',s \in \nat$ with $m\le m'$, let $p \in \FF_b[x]$ with $\deg(p)=m'$ and let $\bsq=(q_1,\ldots,q_s) \in (\FF_b[x])^s$. A higher order polynomial lattice rule over $\FF_b$ is a QMC rule using a higher order polynomial lattice point set $P_{b^m,m'}(\bsq,p)$ consisting of $b^m$ points that are given by
  \begin{align*}
    \bsx_n &:= \left( v_{m'}\left( \frac{n(x)q_1(x)}{p(x)} \right) , \ldots , v_{m'}\left( \frac{n(x)q_s(x)}{p(x)} \right) \right) \in [0,1)^s ,
  \end{align*}
for $0\le n <b^m$. The vector $\bsq$ and the polynomial $p$ are respectively called the {\em generating vector} and the {\em modulus} of $P_{b^m,m'}(\bsq,p)$.
\end{definition}

It is obvious from the above definition that higher order polynomial lattice rules reduce to polynomial lattice rules when $m'=m$. In the remainder of this paper, we refer to polynomial lattice rules as {\em classical} polynomial lattice rules when a clear distinction from higher order polynomial lattice rules is needed. It was shown in \cite{DP07} that there exist good higher order polynomial lattice rules when $m'\ge \alpha m$, which achieve the optimal rate of the mean square worst-case error with respect to a random digital shift for integrands in weighted unanchored Sobolev spaces of smoothness $\alpha$, where $\alpha\ge 2$ is an integer. It was shown later in \cite{BDGP11,BDLNP12} that, in a normed function space different from what is studied in \cite{DP07}, the component-by-component (CBC) construction requires the construction cost of $O(s\alpha N^\alpha \log N)$ operations using $O(N^\alpha)$ memory to obtain good deterministic higher order polynomial lattice rules which achieve the optimal rate of the worst-case error. This large construction cost of higher order polynomial lattice rules is thus the major obstacle for practical applications.

In order to reduce the construction cost significantly while obtaining good point sets, the author considered in \cite{Goxxa,Goxxb} employing the latter construction principle mentioned above, in which classical polynomial lattice point sets are used for interlaced components. It was shown that the construction cost of only $O(s\alpha N \log N)$ operations using $O(N)$ memory is required to obtain good point sets by the CBC construction. This idea of {\em interlaced polynomial lattice rules} stems from \cite{GDxx} where higher order scrambling is considered for randomizing point sets, which is not straightforwardly applicable to higher order polynomial lattice point sets.

In this paper, we attempt to reduce the construction cost of higher order polynomial lattice rules itself. We consider weighted unanchored Sobolev spaces of smoothness $\alpha$ and employ the mean square worst-case error as a quality criterion for constructing higher order polynomial lattice rules. As quadrature points we use higher order polynomial lattice point sets over $\FF_2$ which are randomly digitally shifted and then folded using the tent transformation. By virtue of the tent transformation it is possible to show that there exist good higher order polynomial lattice rules which achieve the optimal rate of the mean square worst-case error when $m'\ge \alpha m/2$. This implies that we can reduce the required degree of the modulus by half as compared to that in \cite{DP07}, where the tent transformation is not applied to randomly digitally shifted higher order polynomial lattice point sets. Thus we are able to restrict the search space of generating vectors significantly. The construction cost required for the CBC construction of higher order polynomial lattice rules becomes of $O(s\alpha N^{\alpha/2} \log N)$ operations using $O(N^{\alpha/2})$ memory. This cost compares favorably with the construction cost of $O(s\alpha N^\alpha \log N)$ operations using $O(N^\alpha)$ memory obtained in \cite{BDLNP12}. We have to note, however, that in \cite{BDLNP12} the worst-case error is employed as a quality criterion so that no randomization is required in contrast to this paper. Our result generalizes the previous result shown in \cite{CDLP07}, in which the degree of smoothness is fixed at 2 and classical polynomial lattice rules are considered.

The remainder of this paper is organized as follows. In Section \ref{sec:pre}, we describe the necessary background and notation, such as randomization of point sets with a random digital shift and the tent transformation, Walsh functions and weighted unanchored Sobolev spaces of smoothness $\alpha$. In Section \ref{sec:error}, we study the mean square worst-case error in weighted unanchored Sobolev spaces of smoothness $\alpha$ of higher order polynomial lattice point sets over $\FF_2$ which are randomly digitally shifted and then folded using the tent transformation. Our aim here is to derive an upper bound on the mean square worst-case error which can be employed as a computable quality criterion of higher order polynomial lattice rules. In Section \ref{sec:existence}, we prove that there exist good higher order polynomial lattice rules which achieve the optimal rate of the mean square worst-case error when $m'\ge \alpha m/2$. In Section \ref{sec:construction}, we investigate the CBC construction as an explicit means of obtaining good higher order polynomial lattice rules. Finally in Section \ref{sec:fast_cbc}, we show how to calculate the quality criterion efficiently and how to obtain the fast CBC construction using the fast Fourier transform in a way analogous to \cite{BDLNP12}.

\section{Preliminaries}\label{sec:pre}

\subsection{Randomization of QMC point sets}\label{subsec:randomization}
We remind that, in the remainder of this paper, we focus on higher order polynomial lattice point sets over $\FF_2$, that is, those with $b=2$ in Definition \ref{def:ho_polynomial_lattice}. As a randomization of the point set $P_{2^m,m'}(\bsq,p)$, we consider first applying a random digital shift to $P_{2^m,m'}(\bsq,p)$ and then folding the resulting point set by using the tent transformation.

We first introduce a random digital shift. Let $P_{2^m,m'}(\bsq,p)=\{\bsx_0,\ldots,\bsx_{2^m-1}\}$ be a higher order polynomial lattice  point set and let $\bssigma=(\sigma_1,\ldots,\sigma_s)$ be such that $\sigma_1,\dots,\sigma_s$ are independently and uniformly distributed in $[0,1)$. Then a randomly digitally shifted higher order polynomial lattice  point set $P_{2^m,m',\bssigma}(\bsq,p)=\{\bsy_0,\ldots, \bsy_{2^m-1}\}$ is given by
  \begin{align*}
    \bsy_n = \bsx_n\oplus \bssigma ,
  \end{align*}
for $0\le n<2^m$, where the operator $\oplus$ denotes the digitwise addition. That is, for $x, y\in [0,1)$ with dyadic expansions $x=\sum_{i=1}^{\infty}x_i 2^{-i}$ and $y=\sum_{i=1}^{\infty}y_i 2^{-i}$ where $x_i,y_i\in \FF_2$, $\oplus$ is defined by
  \begin{align*}
    x\oplus y := \sum_{i=1}^{\infty}\frac{z_i}{2^i} ,\; \text{where}\; z_i=x_i+y_i \pmod 2 .
  \end{align*}
In case of vectors in $[0,1)^s$, the operator $\oplus$ is applied componentwise.

The point set $P_{2^m,m',\bssigma}(\bsq,p)$ is then folded using the tent transformation to obtain the point set $P_{2^m,m',\bssigma,\phi}(\bsq,p)$ which we employ as quadrature points. The tent transformation, or the baker's transformation, was first used in \cite{Hi02} for QMC rules using integration lattices and later studied in \cite{CDLP07} for classical polynomial lattice rules. This transformation is given by the function
  \begin{align*}
    \phi(x) := 1-|2x-1| ,
  \end{align*}
for $x\in [0,1)$. For $\bsx\in [0,1)^s$, this transformation is applied componentwise. Thus the {\em randomly digitally shifted and then folded} higher order polynomial lattice point set $P_{2^m,m',\bssigma,\phi}(\bsq,p)=\{\bsz_0,\ldots, \bsz_{2^m-1}\}$ is given by
  \begin{align*}
    \bsz_n = \phi(\bsy_n)=\phi(\bsx_n\oplus \bssigma) ,
  \end{align*}
for $0\le n<2^m$.

\subsection{Walsh functions}\label{subsec:walsh}
Here we follow the exposition in \cite[Appendix~A]{DP10} to introduce Walsh functions, which will play a major role in the subsequent analysis. We first give the definition of Walsh functions for the one-dimensional case.
\begin{definition}
Let $k\in \nat_0$ with dyadic expansion $k = \kappa_0+\kappa_1 2+\cdots +\kappa_{a}2^{a}$. Then the $k$-th Walsh function $\wal_k: [0,1)\to \{-1,1\}$ is defined as
  \begin{align*}
    \wal_k(x) = (-1)^{\xi_1\kappa_0+\cdots+\xi_{a+1}\kappa_a} ,
  \end{align*}
for $x\in [0,1)$ with dyadic expansion $x=\xi_1 2^{-1}+\xi_2 2^{-2}+\cdots $, that is unique in the sense that infinitely many of the $\xi_i$ are different from 1.
\end{definition}
This definition can be generalized to the multi-dimensional case.

\begin{definition}
For $s\in \nat$, let $\bsx=(x_1,\cdots, x_s)\in [0,1)^s$ and $\bsk=(k_1,\cdots, k_s)\in \nat_0^s$. We define the $\bsk$-th Walsh function $\wal_{\bsk}: [0,1)^s \to \{-1,1\}$ by
  \begin{align*}
    \wal_{\bsk}(\bsx) = \prod_{j=1}^s \wal_{k_j}(x_j) .
  \end{align*}
\end{definition}

In order to introduce the next lemma, we add one more notation and the notion of the dual polynomial lattice of a higher order polynomial lattice point set. For $k\in \nat_0$ with dyadic expansion $k=\kappa_0+\kappa_1 2+\cdots$, we denote by $\rtr_{m'}(k)$ the truncated polynomial given as
  \begin{align*}
    \rtr_{m'}(k)(x)=\kappa_0+\kappa_1 x+\cdots +\kappa_{m'-1}x^{m'-1}.
  \end{align*}
For a higher order polynomial lattice point set $P_{2^m,m'}(\bsq,p)=\{\bsx_0,\ldots,\bsx_{2^m-1}\}$, the dual polynomial lattice of $P_{2^m,m'}(\bsq,p)$, denoted by $\Dcal^{\perp}(\bsq,p)$, is defined as
  \begin{align*}
    \Dcal^{\perp}(\bsq,p) := \{ & \bsk=(k_1,\ldots,k_s)\in \nat_0^s :\\
                            & \rtr_{m'}(k_1)q_1+\cdots +\rtr_{m'}(k_s)q_s\equiv a \pmod p\; \text{with}\; \deg(a)<m'-m \} .
  \end{align*}

The following lemma, which is a special case of \cite[Lemma~4.2]{Di08}, bridges between a higher order polynomial lattice point set $P_{2^m,m'}(\bsq,p)$ and Walsh functions.
\begin{lemma}\label{lemma:dual}
Let $P_{2^m,m'}(\bsq,p)=\{\bsx_0,\ldots,\bsx_{2^m-1}\}$ be a higher order polynomial lattice point set and let $\Dcal^{\perp}(\bsq,p)$ be its dual polynomial lattice. Then we have
  \begin{align*}
    \frac{1}{2^m}\sum_{n=0}^{2^m-1}\wal_{\bsk}(\bsx_n) = \left\{ \begin{array}{ll}
    1 & \text{if} \ \bsk\in \Dcal^{\perp}(\bsq,p) ,  \\
    0 & \text{otherwise} .  \\
    \end{array} \right.
  \end{align*}
\end{lemma}

\subsection{The reproducing kernel Hilbert space}\label{subsec:space}
According to \cite[Section~2.2]{BD09}, we describe the reproducing kernel Hilbert space $\Hcal_{s,\alpha,\bsgamma}$, which will be considered in this paper, for $s,\alpha\in \nat$ with $\alpha\ge 2$ and a set of non-negative real numbers $\bsgamma=(\gamma_u)_{u\subseteq \{1,\ldots,s\}}$. Here $\bsgamma$ are called {\em weights}, whose role is to moderate the importance of different variables or groups of variables in the space $\Hcal_{s,\alpha,\bsgamma}$, and play a major role in analyzing the information complexity that is defined as the minimum number of points $N(\varepsilon,s)$ required to reduce the initial error by a factor $\varepsilon\in (0,1)$, see \cite{SW98}. Our particular interest is to give sufficient conditions on the weights when the bound on $N(\varepsilon,s)$ does not depend on the dimension, or does depend only polynomially on the dimension, see Corollary \ref{cor:tractability}.

Let us consider the one-dimensional unweighted case first. We note that the elements of $\Hcal_{1,\alpha,(1)}$ are defined on the unit interval. The inner product is given by
  \begin{align*}
    \langle f,g\rangle_{\Hcal_{1,\alpha,(1)}}=\sum_{\tau=0}^{\alpha-1}\int_0^1 f^{(\tau)}(x)\rd x\int_0^1 g^{(\tau)}(x)\rd x+\int_0^1 f^{(\alpha)}(x)g^{(\alpha)}(x)\rd x ,
  \end{align*}
where we denote by $f^{(\tau)}$ the $\tau$-th derivative of $f$ and set $f^{(0)}=f$. Let $||f||_{\Hcal_{1,\alpha,(1)}}=\langle f,f\rangle_{\Hcal_{1,\alpha,(1)}}^{1/2}$ be the norm of $f$ associated with $\Hcal_{1,\alpha,(1)}$. We define the function $\Kcal_{1,\alpha,(1)}: [0,1)\times [0,1)\to \RR$ as
  \begin{align*}
    \Kcal_{1,\alpha,(1)}(x,y)=\sum_{\tau=1}^{\alpha}\frac{B_{\tau}(x)B_{\tau}(y)}{(\tau !)^2}+(-1)^{\alpha+1}\frac{B_{2\alpha}(|x-y|)}{(2\alpha)!} ,
  \end{align*}
for $x,y\in [0,1)$, where $B_{\tau}$ denotes the Bernoulli polynomial of degree $\tau$. The reproducing kernel for $\Hcal_{1,\alpha,(1)}$ is given by $1+\Kcal_{1,\alpha,(1)}(x,y)$. That is, for any $f\in \Hcal_{1,\alpha,(1)}$, we have 
  \begin{align*}
    f(x) =\langle f,1+\Kcal_{1,\alpha,(1)}(\cdot,x)\rangle_{\Hcal_{1,\alpha,(1)}} ,
  \end{align*}
for $x\in [0,1)$.

We now consider the multi-dimensional weighted case. The inner product for the $s$-dimensional weighted unanchored Sobolev space $\Hcal_{s,\alpha,\bsgamma}$ is defined by
  \begin{align*}
    & \quad \langle f,g\rangle_{\Hcal_{s,\alpha,\bsgamma}} \\
    & = \sum_{u\subseteq \{1,\ldots,s\}}\gamma_u^{-1}\sum_{v\subseteq u}\sum_{\bstau_{u\setminus v}\in \{1,\ldots, \alpha-1\}^{|u\setminus v|}}\int_{[0,1)^{|v|}} \\
    & \left(\int_{[0,1)^{s-|v|}}f^{(\bstau_{u\setminus v},\bsalpha_v,\bszero)}(\bsx)\rd\bsx_{-v}\right) \left(\int_{[0,1)^{s-|v|}}g^{(\bstau_{u\setminus v},\bsalpha_v,\bszero)}(\bsx)\rd\bsx_{-v}\right) \rd\bsx_{v} ,
  \end{align*}
where we use the following notations: For $\bstau_{u\setminus v}=(\tau_j)_{j\in u\setminus v}$, we denote by $(\bstau_{u\setminus v},\bsalpha_v,\bszero)$ the vector in which the $j$-th component is $\tau_j$ for $j\in u\setminus v$, $\alpha$ for $j\in v$, and 0 for $\{1,\ldots,s\}\setminus u$. For $v\subseteq \{1,\ldots,s\}$, we simply write $-v:=\{1,\ldots,s\}\setminus v$, $\bsx_{v}=(x_j)_{j\in v}$ and $\bsx_{-v}=(x_j)_{j\in -v}$. Further for $u\subseteq \{1,\ldots,s\}$ such that $\gamma_u=0$, we assume that the corresponding inner double sum equals 0 and we set $0/0=0$.

For instance, for the case $\alpha=2$, we can write down the inner product for $\Hcal_{s,2,\bsgamma}$ as
  \begin{align*}
    \langle f,g\rangle_{\Hcal_{s,2,\bsgamma}} & = \sum_{u\subseteq \{1,\ldots,s\}}\gamma_u^{-1}\sum_{v\subseteq u}\int_{[0,1)^{|v|}} \\
    & \quad \left(\int_{[0,1)^{s-|v|}}\frac{\partial^{|u|+|v|}f}{\partial\bsx_u \partial\bsx_v}(\bsx)\rd\bsx_{-v}\right) \left(\int_{[0,1)^{s-|v|}}\frac{\partial^{|u|+|v|}g}{\partial\bsx_u \partial\bsx_v}(\bsx)\rd\bsx_{-v}\right) \rd\bsx_{v} .
  \end{align*}

As in the one-dimensional unweighted case, let $||f||_{\Hcal_{s,\alpha,\bsgamma}}=\langle f,f\rangle_{\Hcal_{s,\alpha,\bsgamma}}^{1/2}$ be the norm of $f$ associated with $\Hcal_{s,\alpha,\bsgamma}$. The reproducing kernel for $\Hcal_{s,\alpha,\bsgamma}$, $\Kcal_{s,\alpha,\bsgamma}:[0,1)^s\times [0,1)^s\to \RR$, becomes
  \begin{align*}
    \Kcal_{s,\alpha,\bsgamma}(\bsx,\bsy) & = \sum_{u\subseteq \{1,\ldots,s\}}\gamma_u \prod_{j\in u}\Kcal_{1,\alpha,(1)}(x_j,y_j) \\
    & = \sum_{u\subseteq \{1,\ldots,s\}}\gamma_u \prod_{j\in u}\left( \sum_{\tau=1}^{\alpha}\frac{B_{\tau}(x_j)B_{\tau}(y_j)}{(\tau !)^2}+(-1)^{\alpha+1}\frac{B_{2\alpha}(|x_j-y_j|)}{(2\alpha)!}\right) ,
  \end{align*}
for $\bsx,\bsy\in [0,1)^s$, in which we set 
  \begin{align*}
    \prod_{j\in \emptyset}\Kcal_{1,\alpha,(1)}(x_j,y_j) = 1 .
  \end{align*}
Then we have for any $f\in \Hcal_{s,\alpha,\bsgamma}$
  \begin{align*}
    f(\bsx) =\langle f,\Kcal_{s,\alpha,\bsgamma}(\cdot,\bsx)\rangle_{\Hcal_{s,\alpha,\bsgamma}} ,
  \end{align*}
for $\bsx\in [0,1)^s$.

\section{Mean square worst-case error}\label{sec:error}
In this section, we study the mean square worst-case error in $\Hcal_{s,\alpha,\bsgamma}$ of higher order polynomial lattice point sets over $\FF_2$ which are randomly digitally shifted and then folded using the tent transformation. Since $\Kcal_{s,\alpha,\bsgamma}\in \Lcal_2([0,1)^{2s})$ and the system $\{\wal_{\bsk}: \bsk\in \nat_0^s\}$ is a complete orthonormal system in $\Lcal_2([0,1)^s)$ for any $s\in \nat$, we can represent the reproducing kernel $\Kcal_{s,\alpha,\bsgamma}(\bsx,\bsy)$ by its Walsh series, that is,
  \begin{align*}
    \Kcal_{s,\alpha,\bsgamma}(\bsx,\bsy) \sim \sum_{\bsk,\bsl\in \nat_0^s}\hat{\Kcal}_{s,\alpha,\bsgamma}(\bsk,\bsl)\wal_{\bsk}(\bsx)\wal_{\bsl}(\bsy) ,
  \end{align*}
where the $(\bsk,\bsl)$-th Walsh coefficient $\hat{\Kcal}_{s,\alpha,\bsgamma}(\bsk,\bsl)$ is given by
  \begin{align*}
    \hat{\Kcal}_{s,\alpha,\bsgamma}(\bsk,\bsl) = \int_{[0,1)^s}\int_{[0,1)^s}\Kcal_{s,\alpha,\bsgamma}(\bsx,\bsy)\wal_{\bsk}(\bsx)\wal_{\bsl}(\bsy)\rd \bsx\rd \bsy .
  \end{align*}
Here we refer to \cite[Appendix A.3]{DP10} for a discussion on the pointwise convergence of the Walsh series.

Using the notations in Subsection \ref{subsec:randomization}, let $P_{2^m,m'}(\bsq,p)=\{\bsx_0,\ldots,\bsx_{2^m-1}\}$ be the higher order polynomial lattice  point set, and let $P_{2^m,m',\bssigma,\phi}(\bsq,p)=\{\bsz_0,\ldots, \bsz_{2^m-1}\}$ be the randomly digitally shifted and then folded higher order polynomial lattice  point set for $\bssigma\in [0,1)^s$. We denote by $e^2(P_{2^m,m',\bssigma,\phi}(\bsq,p),\Hcal_{s,\alpha,\bsgamma})$ the worst-case error in $\Hcal_{s,\alpha,\bsgamma}$ of $P_{2^m,m',\bssigma,\phi}(\bsq,p)$, which is defined by 
  \begin{align*}
    e^2(P_{2^m,m',\bssigma,\phi}(\bsq,p),\Hcal_{s,\alpha,\bsgamma}) := \sup_{\substack{f\in \Hcal_{s,\alpha,\bsgamma} \\ ||f||_{\Hcal_{s,\alpha,\bsgamma}}\le 1}}|I(f)-Q(f;P_{2^m,m',\bssigma,\phi}(\bsq,p))| .
  \end{align*}
We further denote by $\tilde{e}^2(P_{2^m,m'}(\bsq,p),\Hcal_{s,\alpha,\bsgamma})$ the mean square worst-case error in $\Hcal_{s,\alpha,\bsgamma}$ of $P_{2^m,m',\bssigma,\phi}(\bsq,p)$ with respect to $\bssigma$, which is defined by
  \begin{align*}
    \tilde{e}^2(P_{2^m,m'}(\bsq,p),\Hcal_{s,\alpha,\bsgamma}) := \int_{[0,1)^s} e^2(P_{2^m,m',\bssigma,\phi}(\bsq,p),\Hcal_{s,\alpha,\bsgamma}) \rd \bssigma.
  \end{align*}
Before introducing the theorem on $\tilde{e}^2(P_{2^m,m'}(\bsq,p),\Hcal_{s,\alpha,\bsgamma})$, we need to add some more notations: For $k\in \nat$ with dyadic expansion $k=\kappa_0+\kappa_1 2+\kappa_2 2^2+\cdots$, we define the sum-of-digits function as
  \begin{align*}
    \delta(k) := \kappa_0+\kappa_1+\kappa_2+\cdots .
  \end{align*}
Let $\Ecal = \{k\in \nat: \delta(k)\equiv 0 \pmod 2\}$ and $\Ecal_0=\Ecal\cup \{0\}$. Further, for a real number $x$, we denote by $\lfloor x \rfloor$ the largest integer smaller than or equal to $x$. Especially for $k\in \nat_0$ with dyadic expansion $k=\kappa_0+\kappa_1 2+\kappa_2 2^2+\cdots$, we have $\lfloor k/2 \rfloor = \kappa_1+\kappa_2 2+\cdots$. For $\bsk=(k_1,\dots,k_s)\in \nat_0^s$, let $\lfloor \bsk/2 \rfloor=(\lfloor k_1/2 \rfloor, \ldots,\lfloor k_s/2 \rfloor)$.

Then we have the following theorem, which is an obvious adaptation of \cite[Theorems~2 \& 4]{CDLP07} to our current setting.

\begin{theorem}\label{theorem:error1}
Let $P_{2^m,m'}(\bsq,p)$ be a higher order polynomial lattice point set and let $\Dcal^{\perp}(\bsq,p)$ be its dual polynomial lattice. The mean square worst-case error in $\Hcal_{s,\alpha,\bsgamma}$ of $P_{2^m,m',\bssigma,\phi}(\bsq,p)$ with respect to $\bssigma$ is given by
  \begin{align}\label{eq:error1}
    \tilde{e}^2(P_{2^m,m'}(\bsq,p),\Hcal_{s,\alpha,\bsgamma}) = \sum_{\bsk\in (\Ecal_0^s\setminus\{\bszero\})\cap \Dcal^{\perp}(\bsq,p)}\hat{\Kcal}_{s,\alpha,\bsgamma}(\lfloor \bsk/2 \rfloor, \lfloor \bsk/2 \rfloor) .
  \end{align}
\end{theorem}
We also need to mention the mean square initial error, that is, the mean square worst-case error in $\Hcal_{s,\alpha,\bsgamma}$ of $P_{0,m',\bssigma,\phi}(\bsq,p)$. For the empty point set $P_{0}$, the square worst-case error is given as
  \begin{align*}
    e^2(P_{0},\Hcal_{s,\alpha,\bsgamma}) := & \sup_{\substack{f\in \Hcal_{s,\alpha,\bsgamma} \\ ||f||_{\Hcal_{s,\alpha,\bsgamma}}\le 1}}|I(f)| \\
    = & \int_{[0,1)^s}\int_{[0,1)^s}\Kcal_{s,\alpha,\bsgamma}(\bsx,\bsy)\rd \bsx \rd \bsy = \gamma_{\emptyset} .
  \end{align*}
Thus we have
  \begin{align*}
    \tilde{e}^2(P_{0,m'}(\bsq,p),\Hcal_{s,\alpha,\bsgamma}) := & \int_{[0,1)^s} e^2(P_{0,m',\bssigma,\phi}(\bsq,p),\Hcal_{s,\alpha,\bsgamma}) \rd \bssigma \\
    = & \int_{[0,1)^s}\gamma_{\emptyset} \rd \bssigma = \gamma_{\emptyset} .
  \end{align*}

Let $u\subseteq \{1,\ldots,s\}$. For $\bsk_u=(k_j)_{j\in u}\in \nat^{|u|}$, we denote by $(\bsk_u,\bszero)$ the $s$-dimensional vector in which the $j$-th component is $k_j$ for $j\in u$ and 0 for $j\in \{1,\ldots,s\}\setminus u$. For $k\in \nat$ with dyadic expansion $k=2^{a_1-1}+\cdots + 2^{a_v-1}$ such that $a_1>\cdots >a_v>0$, $\mu_{\alpha}(k)$ is defined as 
  \begin{align}\label{eq:mu_alpha}
    \mu_{\alpha}(k)=a_1+\cdots+a_{\min(\alpha,v)} ,
  \end{align}
and $\mu_{\alpha}(\bsk_u)=\prod_{j\in u}\mu_{\alpha}(k_j)$. Using these notations, the next theorem gives an upper bound on $\tilde{e}^2(P_{2^m,m'}(\bsq,p),\Hcal_{s,\alpha,\bsgamma})$.

\begin{theorem}\label{theorem:error2}
Let $P_{2^m,m'}(\bsq,p)$ be a higher order polynomial lattice point set and let $\Dcal^{\perp}(\bsq,p)$ be its dual polynomial lattice. The mean square worst-case error in $\Hcal_{s,\alpha,\bsgamma}$ of $P_{2^m,m',\bssigma,\phi}(\bsq,p)$ is bounded by
  \begin{align}\label{eq:error2}
    \tilde{e}^2(P_{2^m,m'}(\bsq,p),\Hcal_{s,\alpha,\bsgamma}) \le \sum_{\emptyset \ne u \subseteq \{1,\ldots,s\}}\gamma_u D_{\alpha}^{|u|}\sum_{\substack{\bsk_u\in \Ecal^{|u|}\\ (\bsk_u,\bszero) \in \Dcal^{\perp}(\bsq,p)}}2^{-2\mu_{\alpha}(\lfloor\bsk_u/2\rfloor)} ,
  \end{align}
where $D_{\alpha}>0$, which depends only on $\alpha$, is defined as
  \begin{align*}
    D_{\alpha} = \max_{1\le \nu \le \alpha}\left( C'_{\alpha,\nu}+\tilde{C}_{2\alpha}2^{-2(\alpha-\nu)}\right) ,
  \end{align*}
in which $C'_{\alpha,\nu}$ and $\tilde{C}_{2\alpha}$ are respectively given by
  \begin{align*}
    C'_{\alpha,\nu} = \sum_{\tau=\nu}^{\alpha}C_{\tau}^2 2^{-2(\tau-\nu)} ,
  \end{align*}
where $C_{1}=1/2$ and $C_{\tau}=2^{-\tau}(5/3)^{\tau-2}$ for $\tau\ge 2$, and $\tilde{C}_{2\alpha} = 2^{-2\alpha+1}(5/3)^{2\alpha-2}$.
\end{theorem}

\begin{proof}
Using Theorem \ref{theorem:error1} and \cite[Lemma~14, Equation (13) \& Proposition~20]{BD09}, we have
  \begin{align*}
    & \quad \tilde{e}^2(P_{2^m,m'}(\bsq,p),\Hcal_{s,\alpha,\bsgamma}) \\
    & = \sum_{\emptyset \ne u \subseteq \{1,\ldots,s\}}\sum_{\substack{\bsk_u\in \Ecal^{|u|}\\ (\bsk_u,\bszero) \in \Dcal^{\perp}(\bsq,p)}}\hat{\Kcal}_{s,\alpha,\bsgamma}((\lfloor \bsk_u/2 \rfloor,\bszero),(\lfloor \bsk_u/2 \rfloor,\bszero)) \\
    & = \sum_{\emptyset \ne u \subseteq \{1,\ldots,s\}}\gamma_u\sum_{\substack{\bsk_u\in \Ecal^{|u|}\\ (\bsk_u,\bszero) \in \Dcal^{\perp}(\bsq,p)}}\prod_{j\in u}\hat{\Kcal}_{1,\alpha,(1)}(\lfloor k_j/2 \rfloor,\lfloor k_j/2 \rfloor) \\
    & \le \sum_{\emptyset \ne u \subseteq \{1,\ldots,s\}}\gamma_u D_{\alpha}^{|u|}\sum_{\substack{\bsk_u\in \Ecal^{|u|}\\ (\bsk_u,\bszero) \in \Dcal^{\perp}(\bsq,p)}}\prod_{j\in u}2^{-2\mu_{\alpha}(\lfloor k_j/2\rfloor)} .
  \end{align*}
Hence, the result follows.
\end{proof}

In the remainder of this paper, we employ this upper bound on the mean square worst-case error as a quality criterion of higher order polynomial lattice rules. For simplicity of exposition, we shall denote this upper bound by
  \begin{align}\label{eq:criterion}
    B_{\alpha,\bsgamma}(\bsq,p) := \sum_{\emptyset \ne u \subseteq \{1,\ldots,s\}}\gamma_u D_{\alpha}^{|u|}\sum_{\substack{\bsk_u\in \Ecal^{|u|}\\ (\bsk_u,\bszero) \in \Dcal^{\perp}(\bsq,p)}}2^{-2\mu_{\alpha}(\lfloor\bsk_u/2\rfloor)} .
  \end{align}

\section{Existence result}\label{sec:existence}
In this section, we prove the existence of good higher order polynomial lattice rules which achieve the optimal rate of convergence for smooth functions in $\Hcal_{s,\alpha,\bsgamma}$ when $m'\ge \alpha m/2$. Without loss of generality, we can restrict ourselves to considering a vector of polynomials $\bsq=(q_1,\ldots,q_s)\in G_{m'}^s$, where
  \begin{align*}
    G_{m'}=\{q\in \FF_2[x]: \deg(q)<m' \} .
  \end{align*}
This means we have $2^{m's}$ candidates for $\bsq$ in total. The following theorem shows the existence result.

\begin{theorem}\label{theorem:existence}
Let $p\in \FF_2[x]$ be an irreducible polynomial with $\deg(p)=m'$. Then there exists at least one vector of polynomials $\bsq\in G_{m'}^s$, such that for the higher order polynomial lattice point set with generating vector $\bsq$ and modulus $p$ we have
  \begin{align*}
    \tilde{e}^2(P_{2^m,m'}(\bsq,p),\Hcal_{s,\alpha,\bsgamma}) \le \frac{1}{2^{\min(m/\lambda, 4m')}}\left[\sum_{\emptyset \ne u\subseteq \{1,\ldots,s\}}\gamma_u^{\lambda}D_{\alpha}^{\lambda|u|}\left( A_{\alpha,\lambda,1}^{|u|}+A_{\alpha,\lambda,2}^{|u|}\right)\right]^{1/\lambda} ,
  \end{align*}
for any $1/(2\alpha)<\lambda \le 1$, where $A_{\alpha,\lambda,1}$ and $A_{\alpha,\lambda,2}$ are positive and depend only on $\alpha$ and $\lambda$. 
\end{theorem}

\begin{remark}\label{remark:existence}
If $m'\ge \alpha m/2$, we always have $\min(m/\lambda, 4m')=m/\lambda$ and thus
  \begin{align*}
    \tilde{e}^2(P_{2^m,m'}(\bsq,p),\Hcal_{s,\alpha,\bsgamma}) \le \frac{1}{2^{m/\lambda}}\left[\sum_{\emptyset \ne u\subseteq \{1,\ldots,s\}}\gamma_u^{\lambda}D_{\alpha}^{\lambda|u|}\left( A_{\alpha,\lambda,1}^{|u|}+A_{\alpha,\lambda,2}^{|u|}\right)\right]^{1/\lambda} ,
  \end{align*}
for any $1/(2\alpha)<\lambda \le 1$. As we cannot achieve the convergence rate of the mean square worst-case error of order $2^{-2\alpha m}$ in $\Hcal_{s,\alpha,\bsgamma}$ \cite{Sha63}, our result is optimal. Further, the degree of the modulus required to achieve the optimal rate of the mean square worst-case error is reduced by half as compared to that of higher order polynomial lattice rules over $\FF_2$ whose quadrature points are randomly digitally shifted but not folded using the tent transformation, since $m'\ge \alpha m$ is required, see \cite[Theorem~4.4]{DP07}. Our result generalizes the result shown in \cite[Theorem~6]{CDLP07}, where the case of $m'=m$ and $\alpha=2$ is discussed.
\end{remark}

In order to prove Theorem \ref{theorem:existence}, we need the following lemma, which gives the precise values of $A_{\alpha,\lambda,1}$ and $A_{\alpha,\lambda,2}$.

\begin{lemma}\label{lemma:sum-of-digit}
Let $\alpha\ge2$ be an integer and let $\lambda>1/(2\alpha)$ be a real number.
\begin{itemize}
\item We have
  \begin{align*}
    \sum_{k\in \Ecal}2^{-2\lambda\mu_{\alpha}(\lfloor k/2\rfloor)}=A_{\alpha,\lambda,1} ,
  \end{align*}
where
  \begin{align*}
    A_{\alpha,\lambda,1}=\sum_{v=1}^{\alpha-1}\prod_{i=1}^{v}\left( \frac{1}{2^{2\lambda i}-1}\right)+\frac{1}{2^{2\lambda\alpha}-2}\prod_{i=1}^{\alpha-1}\left( \frac{1}{2^{2\lambda i}-1}\right) .
  \end{align*}
\item Let $p \in \FF_2[x]$ be an irreducible polynomial with $\deg(p)=m'$. We have
  \begin{align*}
    \sum_{\substack{k\in \Ecal\\ p\mid \rtr_{m'}(k)}}2^{-2\lambda\mu_{\alpha}(\lfloor k/2\rfloor)}\le \frac{A_{\alpha,\lambda,2}}{2^{4\lambda m'}} ,
  \end{align*}
where
  \begin{align*}
    A_{\alpha,\lambda,2}=\sum_{v=1}^{\lfloor (\alpha-1)/2 \rfloor}\prod_{i=1}^{2v}\left( \frac{2^{2\lambda}}{2^{2\lambda i}-1}\right)+\frac{2^{2\lambda}}{2^{2\lambda\alpha}-2}\prod_{i=1}^{\alpha-1}\left( \frac{2^{2\lambda}}{2^{2\lambda i}-1}\right) .
  \end{align*}
\end{itemize}
\end{lemma}

\begin{proof}
Let us consider the first part. Every positive integer $k\in \Ecal$ must be represented by a dyadic expansion of the form
  \begin{align*}
    k=2^{a_1-1}+\cdots + 2^{a_{2v}-1} ,
  \end{align*}
for $v\in \nat$ such that $a_1>\cdots >a_{2v}>0$, since $\delta(k)$ is even. By collecting $k\in \Ecal$ whose dyadic expansion has the same value of $v$, we have
  \begin{align*}
    \sum_{k\in \Ecal}2^{-2\lambda\mu_{\alpha}(\lfloor k/2\rfloor)}=\sum_{v=1}^{\infty}\sum_{0<a_{2v}<\cdots <a_{1}}2^{-2\lambda \mu_{\alpha}(\lfloor (2^{a_1-1}+\cdots +2^{a_{2v}-1})/2\rfloor)} .
  \end{align*}
On the right-hand side, we have the following. When $a_{2v}=1$, $\lfloor (2^{a_1-1}+\cdots +2^{a_{2v}-1})/2\rfloor=2^{a_1-2}+\cdots +2^{a_{2v-1}-2}$. When $a_{2v}>1$, $\lfloor (2^{a_1-1}+\cdots +2^{a_{2v}-1})/2\rfloor=2^{a_1-2}+\cdots +2^{a_{2v}-2}$. Thus we have
  \begin{align}\label{eq:lemma_sum-of-digit_1}
    \sum_{k\in \Ecal}2^{-2\lambda\mu_{\alpha}(\lfloor k/2\rfloor)} 
    & = \sum_{v=1}^{\infty}\sum_{1<a_{2v-1}<\cdots <a_{1}}2^{-2\lambda \mu_{\alpha}(2^{a_1-2}+\cdots +2^{a_{2v-1}-2})} \nonumber \\
    & \quad + \sum_{v=1}^{\infty}\sum_{1<a_{2v}<\cdots <a_{1}}2^{-2\lambda \mu_{\alpha}(2^{a_1-2}+\cdots +2^{a_{2v}-2})} \nonumber \\
    & = \sum_{v=1}^{\infty}\sum_{0<a_{v}<\cdots <a_{1}}2^{-2\lambda \mu_{\alpha}(2^{a_1-1}+\cdots +2^{a_{v}-1})} \nonumber \\
    & = \sum_{v=1}^{\alpha-1}\sum_{0<a_{v}<\cdots <a_{1}}2^{-2\lambda (a_1+\cdots +a_{v})}+\sum_{v=\alpha}^{\infty}\sum_{0<a_{v}<\cdots <a_{1}}2^{-2\lambda (a_1+\cdots +a_{\alpha})} ,
  \end{align}
where the last equality stems from the definition of $\mu_{\alpha}$. To evaluate the first and second terms of (\ref{eq:lemma_sum-of-digit_1}), we follow a way analogous to the proof of \cite[Lemma~3.1]{BDGP11}. For the first term of (\ref{eq:lemma_sum-of-digit_1}), we have
  \begin{align*}
    \sum_{0<a_{v}<\cdots <a_{1}}2^{-2\lambda (a_1+\cdots +a_{v})} & = \sum_{a_v=1}^{\infty}2^{-2\lambda a_v}\sum_{a_{v-1}=a_v+1}^{\infty}2^{-2\lambda a_{v-1}}\cdots \sum_{a_1=a_2+1}^{\infty}2^{-2\lambda a_1} \\
    & = \frac{1}{2^{2\lambda}-1}\sum_{a_v=1}^{\infty}2^{-2\lambda a_v}\sum_{a_{v-1}=a_v+1}^{\infty}2^{-2\lambda a_{v-1}}\cdots \sum_{a_2=a_3+1}^{\infty}2^{-4\lambda a_2} \\
    & \vdots \\
    & = \prod_{i=1}^{v-1}\left( \frac{1}{2^{2\lambda i}-1}\right)\sum_{a_v=1}^{\infty}2^{-2\lambda va_v} = \prod_{i=1}^{v}\left( \frac{1}{2^{2\lambda i}-1}\right) .
  \end{align*}
For the second term of (\ref{eq:lemma_sum-of-digit_1}), we have
  \begin{align*}
    & \quad \sum_{v=\alpha}^{\infty}\sum_{0<a_{v}<\cdots <a_{1}}2^{-2\lambda (a_1+\cdots +a_{\alpha})} \\
    & = \sum_{v=\alpha}^{\infty}\sum_{a_{v}=1}^{\infty}\sum_{a_{v-1}=a_{v}+1}^{\infty}\cdots\sum_{a_{\alpha+1}=a_{\alpha+2}+1}^{\infty}\sum_{a_{\alpha}=a_{\alpha+1}+1}^{\infty}2^{-2\lambda a_{\alpha}}\sum_{a_{\alpha-1}=a_{\alpha}+1}^{\infty}2^{-2\lambda a_{\alpha-1}}\cdots \sum_{a_1=a_2+1}^{\infty}2^{-2\lambda a_1} \\
    & = \prod_{i=1}^{\alpha}\left( \frac{1}{2^{2\lambda i}-1}\right) \sum_{v=\alpha}^{\infty}\sum_{a_{v}=1}^{\infty}\sum_{a_{v-1}=a_{v}+1}^{\infty}\cdots\sum_{a_{\alpha+1}=a_{\alpha+2}+1}^{\infty}2^{-2\lambda \alpha a_{\alpha+1}} \\
    & = \prod_{i=1}^{\alpha}\left( \frac{1}{2^{2\lambda i}-1}\right) \sum_{v=\alpha}^{\infty}\left( \frac{1}{2^{2\lambda \alpha}-1}\right)^{v-\alpha} \\
    & = \prod_{i=1}^{\alpha}\left( \frac{1}{2^{2\lambda i}-1}\right) \sum_{v=0}^{\infty}\left( \frac{1}{2^{2\lambda \alpha}-1}\right)^{v} =\frac{1}{2^{2\lambda\alpha}-2}\prod_{i=1}^{\alpha-1}\left( \frac{1}{2^{2\lambda i}-1}\right) ,
  \end{align*}
where the last equality requires $\lambda >1/(2\alpha)$. Thus the proof for the first part of this lemma is complete.

Let us move on to the second part. Suppose that $k$ is expressed as $l2^{m'}+k'$ such that $l\in \nat_0$ and $0\le k'<2^{m'}$. If $k'=0$, then we have $\rtr_{m'}(k)=0$ and hence $p\mid \rtr_{m'}(k)$. Otherwise if $k'>0$, $p\nmid \rtr_{m'}(k)$ since $\deg(p)=m'$ and $\deg(\rtr_{m'}(k))<m'$. Thus we only need to sum over $k$ such that $k=l2^{m'}\in \Ecal$ for $l\in \nat_0$. Since $\delta(l2^{m'})=\delta(l)$, we have $l\in \Ecal$ if $l2^{m'}\in \Ecal$. Through this argument, we have
  \begin{align*}
    \sum_{\substack{k\in \Ecal\\ p\mid \rtr_{m'}(k)}}2^{-2\lambda\mu_{\alpha}(\lfloor k/2\rfloor)} = \sum_{l\in \Ecal}2^{-2\lambda\mu_{\alpha}(l2^{m'-1})} .
  \end{align*}
As in the first part, every positive integer $l\in \Ecal$ must be represented by a dyadic expansion of the form $l=2^{a_1-1}+\cdots + 2^{a_{2v}-1}$ for $v\in \nat$ such that $a_1>\cdots >a_{2v}>0$. By collecting $l\in \Ecal$ whose dyadic expansion has the same value of $v$, we have 
  \begin{align}\label{eq:lemma_sum-of-digit_2}
    \sum_{l\in \Ecal}2^{-2\lambda\mu_{\alpha}(l2^{m'-1})} & = \sum_{v=1}^{\infty}\sum_{0<a_{2v}<\cdots <a_1}2^{-2\lambda\mu_{\alpha}(2^{a_1+m'-2}+\cdots +2^{a_{2v}+m'-2})} \nonumber \\
    & = \sum_{v=1}^{\lfloor (\alpha-1)/2\rfloor}\sum_{0<a_{2v}<\cdots <a_1}2^{-2\lambda ((a_1+m'-1)+\cdots +(a_{2v}+m'-1))} \nonumber \\
    & \quad + \sum_{v=\lfloor (\alpha-1)/2\rfloor+1}^{\infty}\sum_{0<a_{2v}<\cdots <a_1}2^{-2\lambda ((a_1+m'-1)+\cdots +(a_{\alpha}+m'-1))} .
  \end{align}
For the first term of (\ref{eq:lemma_sum-of-digit_2}), we have
  \begin{align*}
    & \quad \sum_{0<a_{2v}<\cdots <a_1}2^{-2\lambda ((a_1+m'-1)+\cdots +(a_{2v}+m'-1))} \\
    & = \frac{1}{2^{4\lambda v(m'-1)}}\sum_{0<a_{2v}<\cdots <a_1}2^{-2\lambda (a_1+\cdots +a_{2v})} \\
    & = \frac{1}{2^{4\lambda v(m'-1)}}\prod_{i=1}^{2v}\left( \frac{1}{2^{2\lambda i}-1}\right) \le \frac{1}{2^{4\lambda m'}}\prod_{i=1}^{2v}\left( \frac{2^{2\lambda}}{2^{2\lambda i}-1}\right) .
  \end{align*}
Also for the second term of (\ref{eq:lemma_sum-of-digit_2}), using the result for the first part and considering that $\alpha \ge 2$, we have
  \begin{align*}
    & \quad \sum_{v=\lfloor (\alpha-1)/2\rfloor+1}^{\infty}\sum_{0<a_{2v}<\cdots <a_1}2^{-2\lambda ((a_1+m'-1)+\cdots +(a_{\alpha}+m'-1))} \\
    & \le \sum_{v=\alpha}^{\infty}\sum_{0<a_{v}<\cdots <a_1}2^{-2\lambda ((a_1+m'-1)+\cdots +(a_{\alpha}+m'-1))} \\
    & = \frac{1}{2^{2\lambda\alpha(m'-1)}}\sum_{v=\alpha}^{\infty}\sum_{0<a_{v}<\cdots <a_1}2^{-2\lambda (a_1+\cdots +a_{\alpha})} \\
    & = \frac{1}{2^{2\lambda \alpha(m'-1)}}\cdot \frac{1}{2^{2\lambda\alpha}-2}\prod_{i=1}^{\alpha-1}\left( \frac{1}{2^{2\lambda i}-1}\right) \le \frac{1}{2^{4\lambda m'}}\cdot \frac{2^{2\lambda}}{2^{2\lambda\alpha}-2}\prod_{i=1}^{\alpha-1}\left( \frac{2^{2\lambda}}{2^{2\lambda i}-1}\right) .
  \end{align*}
Inserting these results into (\ref{eq:lemma_sum-of-digit_2}), the proof for the second part of this lemma is complete.
\end{proof}

We are now ready to prove Theorem \ref{theorem:existence}. In the following proof, we shall use the inequality for any sequence of non-negative real numbers $(a_n)_{n\ge 1}$ and any $0<\lambda\le 1$, stating that
  \begin{align}\label{eq:jensen}
    \left( \sum_{n}a_n\right)^{\lambda} \le \sum_{n}a_n^{\lambda} .
  \end{align}

\begin{proof}[Proof of Theorem~\ref{theorem:existence}]
For any $1/(2\alpha)<\lambda\le 1$, there exists at least one vector of polynomials $\bsq\in G_{m'}^s$ for which $B_{\alpha,\bsgamma}^{\lambda}(\bsq,p)$ is smaller than or equal to the average of $B_{\alpha,\bsgamma}^{\lambda}(\tilde{\bsq},p)$ over $\tilde{\bsq}\in G_{m'}^s$, that is,
  \begin{align}\label{eq:averaging}
    B_{\alpha,\bsgamma}^{\lambda}(\bsq,p) \le \frac{1}{2^{m's}}\sum_{\tilde{\bsq}\in G_{m'}^s}B_{\alpha,\bsgamma}^{\lambda}(\tilde{\bsq},p)=:\bar{B}_{\alpha,\bsgamma,\lambda} .
  \end{align}
Applying (\ref{eq:criterion}) and the inequality (\ref{eq:jensen}), we have
  \begin{align}\label{eq:theorem_existence1}
    \bar{B}_{\alpha,\bsgamma,\lambda} & \le \frac{1}{2^{m's}}\sum_{\tilde{\bsq}\in G_{m'}^s}\sum_{\emptyset \ne u\subseteq \{1,\ldots,s\}}\gamma_u^{\lambda}D_{\alpha}^{\lambda|u|}\sum_{\substack{\bsk_u\in \Ecal^{|u|}\\ (\bsk_u,\bszero) \in \Dcal^{\perp}(\tilde{\bsq},p)}}2^{-2\lambda\mu_{\alpha}(\lfloor\bsk_u/2\rfloor)} \nonumber \\
    & = \sum_{\emptyset \ne u\subseteq \{1,\ldots,s\}}\gamma_u^{\lambda}D_{\alpha}^{\lambda|u|}\sum_{\bsk_u\in \Ecal^{|u|}}2^{-2\lambda\mu_{\alpha}(\lfloor\bsk_u/2\rfloor)}\frac{1}{2^{m'|u|}}\sum_{\substack{\tilde{\bsq}_u\in G_{m'}^{|u|}\\ \rtr_{m'}(\bsk_u)\cdot \tilde{\bsq}_u\equiv a \pmod p\\ \deg(a)<m'-m}}1 ,
  \end{align}
where we denote $\rtr_{m'}(\bsk_u)\cdot \tilde{\bsq}_u=\sum_{j\in u}\rtr_{m'}(k_j)\tilde{q}_j$. The innermost sum equals the number of solutions $\tilde{\bsq}_u\in G_{m'}^{|u|}$ such that $\rtr_{m'}(\bsk_u)\cdot \tilde{\bsq}_u\equiv a \pmod p$ with $\deg(a)<m'-m$ and for an irreducible polynomial $p$ with $\deg(p)=m'$. If $\rtr_{m'}(k_j)$ is a multiple of $p$ for all $j\in u$, we always have $\rtr_{m'}(\bsk_u)\cdot \tilde{\bsq}_u\equiv 0 \pmod p$ independently of $\tilde{\bsq}_u$. Thus we have
  \begin{align*}
    \frac{1}{2^{m'|u|}}\sum_{\substack{\tilde{\bsq}_u\in G_{m'}^{|u|}\\ \rtr_{m'}(\bsk_u)\cdot \tilde{\bsq}_u\equiv a \pmod p\\ \deg(a)<m'-m}}1 = 1 .
  \end{align*}
Otherwise if there exists at least one component $\rtr_{m'}(k_j)$ which is not a multiple of $p$, then there are $2^{m'-m}$ possible choices for $a$ such that $\deg(a)<m'-m$, for each of which there are $2^{m'(|u|-1)}$ solutions $\tilde{\bsq}_u$ to $\rtr_{m'}(\bsk_u)\cdot \tilde{\bsq}_u\equiv a \pmod p$. Thus we have
  \begin{align*}
    \frac{1}{2^{m'|u|}}\sum_{\substack{\tilde{\bsq}_u\in G_{m'}^{|u|}\\ \rtr_{m'}(\bsk_u)\cdot \tilde{\bsq}_u\equiv a \pmod p\\ \deg(a)<m'-m}}1 = \frac{1}{2^m} .
  \end{align*}
Inserting these results into (\ref{eq:theorem_existence1}) and using Lemma \ref{lemma:sum-of-digit}, we obtain
  \begin{align*}
    \bar{B}_{\alpha,\bsgamma,\lambda} & \le \sum_{\emptyset \ne u\subseteq \{1,\ldots,s\}}\gamma_u^{\lambda}D_{\alpha}^{\lambda|u|}\left[ \frac{1}{2^m}\sum_{\bsk_u\in \Ecal^{|u|}}2^{-2\lambda\mu_{\alpha}(\lfloor\bsk_u/2\rfloor)}+\sum_{\substack{\bsk_u\in \Ecal^{|u|}\\ p\mid \rtr_{m'}(k_j), \forall j\in u}}2^{-2\lambda\mu_{\alpha}(\lfloor\bsk_u/2\rfloor)}\right] \\
    & = \sum_{\emptyset \ne u\subseteq \{1,\ldots,s\}}\gamma_u^{\lambda}D_{\alpha}^{\lambda|u|}\left[ \frac{1}{2^m}\left(\sum_{k\in \Ecal}2^{-2\lambda\mu_{\alpha}(\lfloor k/2\rfloor)}\right)^{|u|}+\left(\sum_{\substack{k\in \Ecal\\ p\mid \rtr_{m'}(k)}}2^{-2\lambda\mu_{\alpha}(\lfloor k/2\rfloor)}\right)^{|u|}\right] \\
    & \le \sum_{\emptyset \ne u\subseteq \{1,\ldots,s\}}\gamma_u^{\lambda}D_{\alpha}^{\lambda|u|}\left[ \frac{A_{\alpha,\lambda,1}^{|u|}}{2^m}+\left(\frac{A_{\alpha,\lambda,2}}{2^{4\lambda m'}}\right)^{|u|}\right] \\
    & \le \frac{1}{2^{\min(m,4\lambda m')}}\sum_{\emptyset \ne u\subseteq \{1,\ldots,s\}}\gamma_u^{\lambda}D_{\alpha}^{\lambda|u|}\left( A_{\alpha,\lambda,1}^{|u|}+A_{\alpha,\lambda,2}^{|u|}\right) .
  \end{align*}
From (\ref{eq:averaging}), the above bound on $\bar{B}_{\alpha,\bsgamma,\lambda}$ is also a bound on $B_{\alpha,\bsgamma}^{\lambda}(\bsq,p)$. Hence the result follows.
\end{proof}

\section{Component-by-component construction}\label{sec:construction}

Since we have proven the existence of good higher order polynomial lattice rules in the last section, it is desirable to have an explicit means of constructing such rules. For this purpose, we investigate the CBC construction in this section. In the following, we write $\bsq_{\tau}=(q_1,\ldots,q_{\tau})\in G_{m'}^{\tau}$ for $\tau\in \nat_0$, where $\bsq_{0}$ denotes the empty set, and define
  \begin{align*}
    B_{\alpha,\bsgamma}(\bsq_{\tau},p) := \sum_{\emptyset \ne u \subseteq \{1,\ldots,\tau\}}\gamma_u D_{\alpha}^{|u|}\sum_{\substack{\bsk_u\in \Ecal^{|u|}\\ (\bsk_u,\bszero) \in \Dcal^{\perp}(\bsq_{\tau},p)}}2^{-2\mu_{\alpha}(\lfloor\bsk_u/2\rfloor)} ,
  \end{align*}
for $1\le \tau\le s$, where we denote by $(\bsk_u,\bszero)$ the $\tau$-dimensional vector in which the $j$-th component is $k_j$ for $j\in u$ and 0 for $j\in \{1,\ldots,\tau\}\setminus u$. The CBC construction proceeds as follows.

\begin{algorithm}\label{algorithm:cbc}
For $s,m,m',\alpha\in \nat$ with $\alpha\ge 2$ and $m'\ge m$, and for a set of weights $\bsgamma=(\gamma_u)_{u\subseteq \{1,\ldots,s\}}$,
	\begin{enumerate}
		\item Choose an irreducible polynomial $p\in \FF_2[x]$ with $\deg(p)=m'$.
		\item For $\tau=1,\ldots, s$, find $q_{\tau}$ which minimizes $B_{\alpha,\bsgamma}((\bsq_{\tau-1},\tilde{q}_{\tau}),p)$ as a function of $\tilde{q}_{\tau}\in G_{m'}$.
	\end{enumerate}
\end{algorithm}
The next theorem gives an upper bound on the mean square worst-case error in $\Hcal_{\tau,\alpha,\bsgamma}$ for $\bsq_{\tau}\in G_{m'}^{\tau}$ obtained according to Algorithm \ref{algorithm:cbc} for $1\le \tau\le s$.

\begin{theorem}\label{theorem:cbc}
For $1\le \tau\le s$, let $p$ and $\bsq_{\tau}\in G_{m'}^{\tau}$ be obtained according to Algorithm \ref{algorithm:cbc}. Then we have
  \begin{align*}
    B_{\alpha,\bsgamma}(\bsq_{\tau},p) \le \frac{1}{2^{\min( m/\lambda, 4m')}}\left[\sum_{\emptyset \ne u\subseteq \{1,\ldots,\tau\}}\gamma_u^{\lambda}D_{\alpha}^{\lambda|u|}A_{\alpha,\lambda}^{|u|}\right]^{1/\lambda} ,
  \end{align*}
for any $1/(2\alpha)< \lambda \le 1$, where $A_{\alpha,\lambda}=A_{\alpha,\lambda,1}+A_{\alpha,\lambda,2}$ in which $A_{\alpha,\lambda,1}$ and $A_{\alpha,\lambda,2}$ are given as in Lemma \ref{lemma:sum-of-digit}.
\end{theorem}

\begin{remark}\label{remark:cbc}
Let $p$ and $\bsq\in G_{m'}^s$ be obtained according to Algorithm \ref{algorithm:cbc}. When $m'\ge \alpha m/2$, we always have $\min(m/\lambda, 4m')=m/\lambda$, so that
  \begin{align*}
    \tilde{e}^2(P_{2^m,m'}(\bsq,p),\Hcal_{s,\alpha,\bsgamma}) \le \frac{1}{2^{m/\lambda}}\left[\sum_{\emptyset \ne u\subseteq \{1,\ldots,s\}}\gamma_u^{\lambda}D_{\alpha}^{\lambda|u|}A_{\alpha,\lambda}^{|u|}\right]^{1/\lambda} ,
  \end{align*}
for any $1/(2\alpha)<\lambda \le 1$. Thus, Theorem~\ref{theorem:cbc} shows that the CBC construction can find good higher order polynomial lattice rules which achieve the optimal rate of convergence in the space $\Hcal_{s,\alpha,\bsgamma}$. 
\end{remark}

\begin{proof}[Proof of Theorem~\ref{theorem:cbc}]
We prove the theorem by induction. Let us consider the case $\tau=1$ first. There exists at least one polynomial $q_1\in G_{m'}$ for which $B_{\alpha,\bsgamma}^{\lambda}(q_1,p)$ is smaller than or equal to the average of $B_{\alpha,\bsgamma}^{\lambda}(\tilde{q}_1,p)$ over $\tilde{q}_1\in G_{m'}$. Thus, in exactly the same way with the proof of Theorem \ref{theorem:existence} for the case $s=1$, we have for any $1/(2\alpha)< \lambda \le 1$
  \begin{align*}
    B_{\alpha,\bsgamma}^{\lambda}(q_1,p) \le & \frac{1}{2^{\min(m,4\lambda m')}}\gamma_{\{1\}}^{\lambda}D_{\alpha}^{\lambda}\left( A_{\alpha,\lambda,1}+A_{\alpha,\lambda,2}\right) .
  \end{align*}
Hence, the result follows for $\tau=1$.

Next we suppose that for a given $\tau$ with $1\le \tau<s$, the inequality
  \begin{align*}
    B_{\alpha,\bsgamma}(\bsq_{\tau},p) \le \frac{1}{2^{\min( m/\lambda, 4m')}}\left[\sum_{\emptyset \ne u\subseteq \{1,\ldots,\tau\}}\gamma_u^{\lambda}D_{\alpha}^{\lambda|u|}A_{\alpha,\lambda}^{|u|}\right]^{1/\lambda} 
  \end{align*}
holds true for any $1/(2\alpha)<\lambda \le 1$. Then we have
  \begin{align*}
    & \quad B_{\alpha,\bsgamma}((\bsq_{\tau},\tilde{q}_{\tau+1}),p) \\
    & = \sum_{\emptyset \ne u\subseteq \{1,\ldots,\tau+1\}}\gamma_u D_{\alpha}^{|u|}\sum_{\substack{\bsk_u\in \Ecal^{|u|}\\ (\bsk_u,\bszero) \in D^{\perp}((\bsq_{\tau},\tilde{q}_{\tau+1}),p)}}2^{-2\mu_{\alpha}(\lfloor\bsk_u/2\rfloor)} \\
    & = \sum_{\emptyset \ne u\subseteq \{1,\ldots,\tau\}}\gamma_u D_{\alpha}^{|u|}\sum_{\substack{\bsk_u\in \Ecal^{|u|}\\ (\bsk_u,\bszero) \in D^{\perp}(\bsq_{\tau},p)}}2^{-2\mu_{\alpha}(\lfloor\bsk_u/2\rfloor)} \\
    & \quad + \sum_{u\subseteq \{1,\ldots,\tau\}}\gamma_{u\cup \{\tau+1\}} D_{\alpha}^{|u|+1}\sum_{\substack{\bsk_{u\cup\{\tau+1\}}\in \Ecal^{|u|+1}\\ (\bsk_{u\cup \{\tau+1\}},\bszero) \in D^{\perp}((\bsq_{\tau},\tilde{q}_{\tau+1}),p)}}2^{-2\mu_{\alpha}(\lfloor \bsk_{u\cup\{\tau+1\}}/2\rfloor)} \\
    & =: B_{\alpha,\bsgamma}(\bsq_{\tau},p)+\theta(\bsq_{\tau},\tilde{q}_{\tau+1}) ,
  \end{align*}
where we denote by $\theta(\bsq_{\tau},\tilde{q}_{\tau+1})$ the second term in the last equality. In Algorithm \ref{algorithm:cbc}, we choose $q_{\tau+1}$ which minimizes $\theta(\bsq_{\tau},\tilde{q}_{\tau+1})$ among $\tilde{q}_{\tau+1}\in G_{m'}$, since the dependence of $B_{\alpha,\bsgamma}((\bsq_{\tau},\tilde{q}_{\tau+1}),p)$ on $\tilde{q}_{\tau+1}$ appears only in $\theta(\bsq_{\tau},\tilde{q}_{\tau+1})$. Following the averaging argument as in the case $\tau=1$ and using (\ref{eq:jensen}), we have for any $1/(2\alpha)< \lambda \le 1$
  \begin{align*}
    \theta^{\lambda}(\bsq_{\tau},q_{\tau+1}) & \le \frac{1}{2^{m'}}\sum_{\tilde{q}_{\tau+1}\in G_{m'}}\theta^{\lambda}(\bsq_{\tau},\tilde{q}_{\tau+1}) \\
    & \le \frac{1}{2^{m'}}\sum_{\tilde{q}_{\tau+1}\in G_{m'}}\sum_{u\subseteq \{1,\ldots,\tau\}}\gamma_{u\cup \{\tau+1\}}^{\lambda}D_{\alpha}^{\lambda(|u|+1)}\sum_{\substack{\bsk_{u\cup \{\tau+1\}}\in \Ecal^{|u|+1}\\ (\bsk_{u\cup \{\tau+1\}},\bszero) \in D^{\perp}((\bsq_{\tau},\tilde{q}_{\tau+1}),p)}}2^{-2\lambda \mu_{\alpha}(\lfloor \bsk_{u\cup \{\tau+1\}}/2\rfloor)} \\
    & = \sum_{u\subseteq \{1,\ldots,\tau\}}\gamma_{u\cup \{\tau+1\}}^{\lambda}D_{\alpha}^{\lambda(|u|+1)}\sum_{\bsk_u\in \Ecal^{|u|}}2^{-2\lambda \mu_{\alpha}(\lfloor \bsk_u/2\rfloor)}\sum_{k_{\tau+1}\in \Ecal}2^{-2\lambda \mu_{\alpha}(\lfloor k_{\tau+1}/2\rfloor)} \\
    & \quad \times \frac{1}{2^{m'}}\sum_{\substack{ \tilde{q}_{\tau+1}\in G_{m'}\\ \rtr_{m'}(\bsk_u)\cdot \bsq_u + \rtr_{m'}(k_{\tau+1})\cdot \tilde{q}_{\tau+1}\equiv a\pmod p \\ \deg(a)<m'-m}}1 .
  \end{align*}
We remind that $p$ is an irreducible polynomial over $\FF_2$ with $\deg(p)=m'$. If $\rtr_{m'}(k_{\tau+1})$ is a multiple of $p$,
  \begin{align*}
    \rtr_{m'}(\bsk_u)\cdot \bsq_u + \rtr_{m'}(k_{\tau+1})\cdot \tilde{q}_{\tau+1}\equiv \rtr_{m'}(\bsk_u)\cdot \bsq_u \pmod p .
  \end{align*}
Thus, the innermost sum equals $2^{m'}$ for $(\bsk_u,\bszero) \in D^{\perp}(\bsq_{\tau},p)$, and equals 0 otherwise. If $\rtr_{m'}(k_{\tau+1})$ is not a multiple of $p$, there are $2^{m'-m}$ possible choices for $a$ such that $\deg(a)<m'-m$, for each of which there exists at most one solution $\tilde{q}_{\tau+1}$ to $\rtr_{m'}(k_{\tau+1})\cdot \tilde{q}_{\tau+1}\equiv a-\rtr_{m'}(\bsk_u)\cdot \bsq_u \pmod p$. Thus, the innermost sum is bounded above by $2^{m'-m}$. From these observations, we obtain
  \begin{align*}
    & \quad \theta^{\lambda}(\bsq_{\tau},q_{\tau+1}) \\
    & \le \sum_{u\subseteq \{1,\ldots,\tau\}}\gamma_{u\cup \{\tau+1\}}^{\lambda}D_{\alpha}^{\lambda(|u|+1)}\sum_{\substack{\bsk_u\in \Ecal^{|u|}\\ (\bsk_u,\bszero)\in D^{\perp}(\bsq_{\tau}^*,p)}}2^{-2\lambda \mu_{\alpha}(\lfloor \bsk_u/2\rfloor)}\sum_{\substack{k_{\tau+1}\in \Ecal\\ p\mid \rtr_{m'}(k_{\tau+1})}}2^{-2\lambda \mu_{\alpha}(\lfloor k_{\tau+1}/2\rfloor)} \\
    & \quad + \frac{1}{2^m}\sum_{u\subseteq \{1,\ldots,\tau\}}\gamma_{u\cup \{\tau+1\}}^{\lambda}D_{\alpha}^{\lambda(|u|+1)}\sum_{\bsk_u\in \Ecal^{|u|}}2^{-2\lambda \mu_{\alpha}(\lfloor \bsk_u/2\rfloor)}\sum_{\substack{k_{\tau+1}\in \Ecal\\ p\nmid \rtr_{m'}(k_{\tau+1})}}2^{-2\lambda \mu_{\alpha}(\lfloor k_{\tau+1}/2\rfloor)} \\
    & \le \sum_{u\subseteq \{1,\ldots,\tau\}}\gamma_{u\cup \{\tau+1\}}^{\lambda}D_{\alpha}^{\lambda(|u|+1)}\sum_{\bsk_u\in \Ecal^{|u|}}2^{-2\lambda \mu_{\alpha}(\lfloor \bsk_u/2\rfloor)} \\
    & \quad \times \left[ \sum_{\substack{k_{\tau+1}\in \Ecal\\ p\mid \rtr_{m'}(k_{\tau+1})}}2^{-2\lambda \mu_{\alpha}(\lfloor k_{\tau+1}/2\rfloor)}+\frac{1}{2^m}\sum_{k_{\tau+1}\in \Ecal}2^{-2\lambda \mu_{\alpha}(\lfloor k_{\tau+1}/2\rfloor)}\right] \\
    & \le \sum_{u\subseteq \{1,\ldots,\tau\}}\gamma_{u\cup \{\tau+1\}}^{\lambda}D_{\alpha}^{\lambda(|u|+1)}A_{\alpha,\lambda,1}^{|u|}\left( \frac{A_{\alpha,\lambda,1}}{2^m}+\frac{A_{\alpha,\lambda,2}}{2^{4\lambda m'}}\right) \\
    & \le \frac{1}{2^{\min(m,4\lambda m')}}\sum_{u\subseteq \{1,\ldots,\tau\}}\gamma_{u\cup \{\tau+1\}}^{\lambda}D_{\alpha}^{\lambda(|u|+1)}A_{\alpha,\lambda}^{|u|+1},
  \end{align*}
where we use Lemma \ref{lemma:sum-of-digit} in the third inequality. Thus we have
  \begin{align*}
    B_{\alpha,\bsgamma}^{\lambda}(\bsq_{\tau+1},p) & \le B_{\alpha,\bsgamma}^{\lambda}(\bsq_{\tau},p)+\theta^{\lambda}(\bsq_{\tau},q_{\tau+1}) \\
    & \le \frac{1}{2^{\min(m,4\lambda m')}}\sum_{\emptyset \ne u\subseteq \{1,\ldots,\tau\}}\gamma_u^{\lambda}D_{\alpha}^{\lambda|u|}A_{\alpha,\lambda}^{|u|} \\
    & \quad + \frac{1}{2^{\min(m,4\lambda m')}}\sum_{u\subseteq \{1,\ldots,\tau\}}\gamma_{u\cup \{\tau+1\}}^{\lambda}D_{\alpha}^{\lambda(|u|+1)}A_{\alpha,\lambda}^{|u|+1} \\
    & = \frac{1}{2^{\min(m,4\lambda m')}}\sum_{\emptyset \ne u\subseteq \{1,\ldots,\tau+1\}}\gamma_u^{\lambda}D_{\alpha}^{\lambda|u|}A_{\alpha,\lambda}^{|u|} .
  \end{align*}
Hence, the result follows.
\end{proof}

In the following, we briefly discuss the information complexity. As mentioned in Subsection \ref{subsec:space}, the information complexity is defined as the minimum number of points $N(\varepsilon,s)$ required to reduce the initial error by a factor $\varepsilon\in (0,1)$. If there exist non-negative $C,a,b$ such that
  \begin{align*}
    N(\varepsilon,s) \le C s^a \varepsilon^{-b}.
  \end{align*}
for any $s\in \nat$ and for all $\varepsilon\in (0,1)$, we say that multivariate integration in the sequence of spaces $\{\Hcal_{s,\alpha,\bsgamma}\}_{s\ge 1}$ is QMC-tractable. Especially when the above inequality holds with $a=0$, we say that multivariate integration in the sequence of spaces $\{\Hcal_{s,\alpha,\bsgamma}\}_{s\ge 1}$ is strong QMC-tractable. The infimum $a$ and $b$ are called the $s$-exponent and the $\varepsilon$-exponent, respectively.

In the following corollary of Theorem \ref{theorem:cbc}, we assume that $m' \ge \alpha m/2$ and give sufficient conditions on the weights under which multivariate integration in the sequence of spaces $\{\Hcal_{s,\alpha,\bsgamma}\}_{s\ge 1}$ is QMC-tractable or strong QMC-tractable. Since the proof is straightforward, we omit it.
\begin{corollary}\label{cor:tractability}
Let $s,m,m',\alpha\in \nat$ with $\alpha\ge 2$ and $m'\ge \alpha m/2$ and let $\bsgamma=(\gamma_u)_{u\in \{1,\ldots,s\}}$. Let $p$ and $\bsq\in G_{m'}^s$ be found according to Algorithm \ref{algorithm:cbc}. We define
  \begin{align*}
    T_{\lambda,a}:=\limsup_{s\to \infty}\left[\frac{1}{s^a}\sum_{\emptyset \ne u\subseteq \{1,\ldots,s\}}\gamma_u^{\lambda}D_{\alpha}^{\lambda|u|}A_{\alpha,\lambda}^{|u|}\right] ,
  \end{align*}
for $a\ge 0$ and $1/(2\alpha)<\lambda \le 1$.
\begin{enumerate}
\item Assume $T_{\lambda,0}<\infty$ for some $1/(2\alpha)<\lambda\le 1$. Then multivariate integration in the sequence of spaces $\{\Hcal_{s,\alpha,\bsgamma}\}_{s\ge 1}$ is strong QMC-tractable. The $\varepsilon$-exponent is at most $\lambda$.
\item Assume $T_{\lambda,a}<\infty$ for some $1/(2\alpha)<\lambda\le 1$ and $a>0$. Then multivariate integration in the sequence of spaces $\{\Hcal_{s,\alpha,\bsgamma}\}_{s\ge 1}$ is QMC-tractable. The $s$-exponent is at most $a$ and the $\varepsilon$-exponent is at most $\lambda$.
\end{enumerate}
\end{corollary}

\section{Fast construction algorithm}\label{sec:fast_cbc}
Finally in this section, we show how to calculate the quality criterion $B_{\alpha,\bsgamma}(\bsq,p)$ efficiently and how to obtain the fast CBC construction using the fast Fourier transform in a way analogous to \cite{BDLNP12}. We note that the use of the fast Fourier transform for obtaining the fast CBC construction was first studied in \cite{NC06a,NC06b}. In the second part, we focus on the case of product weights, that is, $\gamma_u=\prod_{j\in u} \gamma_j$ for all $u\subseteq \{1,\ldots,s\}$.

\subsection{Efficient calculation of the quality criterion}
Applying Lemma \ref{lemma:dual} to the expression of $B_{\alpha,\bsgamma}(\bsq,p)$ as shown in the right-hand side of (\ref{eq:error2}), we have the following corollary.
\begin{corollary}\label{corollary:quality_criterion}
Let $P_{2^m,m'}(\bsq,p)=\{\bsx_0,\ldots,\bsx_{2^m-1}\}$ be a higher order polynomial lattice point set with generating vector $\bsq$ and modulus $p$. The quality criterion $B_{\alpha,\bsgamma}(\bsq,p)$ is expressed by
  \begin{align*}
    B_{\alpha,\bsgamma}(\bsq,p) = \frac{1}{2^m}\sum_{n=0}^{2^m-1}\sum_{\emptyset \ne u\subseteq \{1,\ldots,s\}}\gamma_u D_{\alpha}^{|u|}\prod_{j\in u}\omega_{\alpha}(x_{n,j}) ,
  \end{align*}
where we define for any $x\in [0,1)$
  \begin{align*}
    \omega_{\alpha}(x)=\sum_{k\in \Ecal}2^{-2\mu_{\alpha}(\lfloor k/2\rfloor)}\wal_{k}(x) .
  \end{align*}
\end{corollary}

The difficulty in calculating $B_{\alpha,\bsgamma}(\bsq,p)$ stems from the fact that $\omega_{\alpha}(x)$ is an infinite sum over $k\in \Ecal$. In the following, we discuss how to calculate $\omega_{\alpha}(x)$ efficiently under the assumption that $x\in [0,1)$ is given in the form $l2^{-m'}$ for $m'>1$, $m' \in \nat$ and $0\le l<2^{m'}$. We note that this is a natural assumption since according to Definition \ref{def:ho_polynomial_lattice} the components of the points $\bsx_n$ are of exactly such a form. Now as in \cite[Section~4]{BDLNP12}, we show that the value of $\omega_{\alpha}(x)$ can be computed in at most $O(\alpha m')$ operations.

\begin{theorem}\label{theorem:omega_calculation}
Let $x\in [0,1)$ be given in the form $l2^{-m'}$ for $m'>1$, $m' \in \nat$ and $0\le l<2^{m'}$. We denote the dyadic expansion of $x$ by $x=\xi_1/2+\cdots+\xi_{m'}/2^{m'}$, where $\xi_1,\ldots,\xi_{m'}\in \FF_2$, and set $\xi_{m'+1}=\xi_{m'+2}=\cdots=0$. Then $\omega_{\alpha}(x)$ can be calculated as follows: we define the following vectors
  \begin{align*}
    & \bsU = (U_0,U_1,\ldots,U_{\alpha-1}) ,\\
    & \tilde{\bsU}(\xi_1) = (\tilde{U}_0(\xi_1),\tilde{U}_1(\xi_1),\ldots,\tilde{U}_{\alpha-1}(\xi_1)) ,\\
    & \bsV(x) = (V_1(x),\ldots,V_{\alpha-1}(x)) ,\\
    & \tilde{\bsV}(x) = (\tilde{V}_{\alpha}(x),\ldots,\tilde{V}_1(x)) ,
  \end{align*}
where we denote
  \begin{align*}
    U_0 = & 1 ,\\
    U_t = & \frac{1}{2^{2t(m'-1)}}\prod_{i=1}^{t}\frac{1}{2^{2i}-1} ,
  \end{align*}
for $1\le t\le \alpha-1$,
  \begin{align*}
    & \tilde{U}_t(\xi_1) = \sum_{v=t}^{\alpha-1}(-1)^{v\xi_1}U_{v-t} ,
  \end{align*}
for $0\le t\le \alpha-1$ and $\xi_1\in \FF_2$, and further
  \begin{align*}
    & V_t(x) = \sum_{0<a_t<\cdots < a_1<m'}\prod_{i=1}^{t}2^{-2a_i}(-1)^{\xi_{a_i+1}} ,\\
    & \tilde{V}_t(x) = \sum_{0<a_t<\cdots < a_1<m'}2^{a_t-1}[ \phi(x)<2^{-a_t+1}]\prod_{i=1}^{t}2^{-2a_i}(-1)^{\xi_{a_i+1}}, 
  \end{align*}
for $1\le t\le \alpha-1$ and $1\le t\le \alpha$, respectively. In the last expression, the value of $[ \phi(x)<2^{-a_t+1}]$ equals 1 if $\phi(x)<2^{-a_t+1}$, and 0 otherwise. Using these notations, we have
  \begin{align*}
    \omega_{\alpha}(x) = \tilde{\bsU}_{1:\alpha-1}(\xi_1)\cdot \bsV(x)+(\tilde{U}_{0}(\xi_1)-1)+(-1)^{\alpha\xi_1}\bsU\cdot \tilde{\bsV}(x) ,
  \end{align*}
for $x\ne 0$, where $\bsa\cdot \bsb$ denotes the dot product of two vectors $\bsa$ and $\bsb$, and $\tilde{\bsU}_{1:\alpha-1}$ is the vector of the last $\alpha-1$ components of $\tilde{\bsU}$. For $x= 0$, we have
  \begin{align*}
    \omega_{\alpha}(0) = \sum_{v=1}^{\alpha-1}\prod_{i=1}^{v}\left(\frac{1}{2^{2i}-1}\right)+\frac{1}{2^{2\alpha}-2}\prod_{i=1}^{\alpha-1}\left(\frac{1}{2^{2i}-1}\right).
  \end{align*}
\end{theorem}

\begin{remark}
$V_t(x)$ and $\tilde{V}_t(x)$ can be calculated as
  \begin{align*}
    V_t(x) = \sum_{a_1=t}^{m'-1}2^{-2a_1}(-1)^{\xi_{a_1+1}}\sum_{a_2=t-1}^{a_1-1}2^{-2a_2}(-1)^{\xi_{a_2+1}}\cdots \sum_{a_t=1}^{a_{t-1}-1}2^{-2a_t}(-1)^{\xi_{a_t+1}} ,
  \end{align*}
and
  \begin{align*}
    \tilde{V}_t(x) = \sum_{a_1=t}^{m'-1}2^{-2a_1}(-1)^{\xi_{a_1+1}}\sum_{a_2=t-1}^{a_1-1}2^{-2a_2}(-1)^{\xi_{a_2+1}}\cdots \sum_{a_t=1}^{a_{t-1}-1}2^{a_t-1}[ \phi(x)<2^{-a_t+1}]2^{-2a_t}(-1)^{\xi_{a_t+1}},
  \end{align*}
respectively. By using these forms, the vectors $\bsV(x)$ and $\tilde{\bsV}(x)$ can be computed in $O(\alpha m')$ operations according to \cite[Algorithm~4]{BDLNP12}. Thus the value of $\omega_{\alpha}(x)$ can be also computed in at most $O(\alpha m')$ operations. 
\end{remark}

In the following proof of Theorem~\ref{theorem:omega_calculation}, we define $V_{0}(x):=1$ for any $x\in [0,1)$.

\begin{proof}[Proof of Theorem~\ref{theorem:omega_calculation}]
Let us consider the case $x=0$ first. From the definition of $\omega_{\alpha}$, we have
  \begin{align*}
    \omega_{\alpha}(0) = \sum_{k\in \Ecal}2^{-2\mu_{\alpha}(\lfloor k/2\rfloor)}=A_{\alpha,1,1} ,
  \end{align*}
where $A_{\alpha,1,1}$ is given in Lemma~\ref{lemma:sum-of-digit}, which proves the result.

Let us consider the case $x\ne 0$ next. As in the proof of Lemma \ref{lemma:sum-of-digit}, every positive integer $k\in \Ecal$ can be represented by a dyadic expansion of the form $k=2^{a_1-1}+\cdots+2^{a_{2v}-1}$ for $v\in \nat$ such that $a_1>\cdots >a_{2v}>0$. Under the assumption that $x\in [0,1)$ is given in the form $l2^{-m'}$ for $m'>1$, $m' \in \nat$ and $0\le l<2^{m'}$, we have 
  \begin{align*}
    \omega_{\alpha}(x) & = \sum_{k\in \Ecal}2^{-2\mu_{\alpha}(\lfloor k/2\rfloor)}\wal_{k}(x) \\
    & = \sum_{v=1}^{\infty}\sum_{0<a_{2v}<\cdots <a_1}2^{-2\mu_{\alpha}(\lfloor (2^{a_1-1}+\cdots +2^{a_{2v}-1})/2\rfloor )}(-1)^{\xi_{a_1}+\cdots+\xi_{a_{2v}}} \\
    & = \sum_{v=1}^{\infty}\sum_{1<a_{2v-1}<\cdots <a_1}2^{-2\mu_{\alpha}(2^{a_1-2}+\cdots +2^{a_{2v-1}-2})}(-1)^{\xi_{a_1}+\cdots+\xi_{a_{2v-1}}+\xi_1} \\
    & \quad + \sum_{v=1}^{\infty}\sum_{1<a_{2v}<\cdots <a_1}2^{-2\mu_{\alpha}(2^{a_1-2}+\cdots +2^{a_{2v}-2})}(-1)^{\xi_{a_1}+\cdots+\xi_{a_{2v}}} \\
    & = \sum_{v=1}^{\infty}(-1)^{v\xi_1}\sum_{0<a_v<\cdots <a_1}2^{-2\mu_{\alpha}(2^{a_1-1}+\cdots +2^{a_v-1})}(-1)^{\xi_{a_1+1}+\cdots+\xi_{a_v+1}} ,
  \end{align*}
where we obtain the third equality by considering the cases $a_{2v}=1$ and $a_{2v}>1$ separately. From the definition of $\mu_{\alpha}$ as in (\ref{eq:mu_alpha}) we obtain
  \begin{align}\label{eq:omega}
    \omega_{\alpha}(x) & = \sum_{v=1}^{\alpha-1}(-1)^{v\xi_1}\sum_{0<a_v<\cdots <a_1}\prod_{i=1}^{v}2^{-2a_i}(-1)^{\xi_{a_i+1}} \nonumber \\
    & \quad + \sum_{v=\alpha}^{\infty}(-1)^{v\xi_1}\sum_{0<a_v<\cdots <a_1}\prod_{i=1}^{\alpha}2^{-2a_i}(-1)^{\xi_{a_i+1}}\prod_{i'=\alpha+1}^{v}(-1)^{\xi_{a_{i'}+1}}  .
  \end{align}
The second term on the right-hand side of (\ref{eq:omega}) can be rewritten as
  \begin{align*}
    & \quad (-1)^{\alpha \xi_1}\sum_{v=\alpha}^{\infty}\sum_{0< a_v< \cdots <a_1}\prod_{i=1}^{\alpha}2^{-2a_i}(-1)^{\xi_{a_i+1}}\prod_{i'=\alpha+1}^{v}(-1)^{\xi_{a_{i'}+1}+\xi_1} \\
    & = (-1)^{\alpha\xi_1}\sum_{0<a_\alpha <\cdots <a_1}\prod_{i=1}^{\alpha}2^{-2a_i}(-1)^{\xi_{a_i+1}}\\
    & \quad \times  \sum_{v=\alpha}^{\infty}\sum_{0< a_v< \cdots <a_{\alpha+1}<a_{\alpha}}\prod_{i'\in \{a_{v},\ldots,a_{\alpha+1}\}}(-1)^{\xi_{i'+1}+\xi_1}\prod_{i''\in \{1,\ldots,a_{\alpha}-1\}\setminus \{a_{v},\ldots,a_{\alpha+1}\}}1  \\
    & = (-1)^{\alpha\xi_1}\sum_{0<a_\alpha <\cdots <a_1}\prod_{i=1}^{\alpha}2^{-2a_i}(-1)^{\xi_{a_i+1}}\sum_{u\subseteq \{1,\ldots,a_{\alpha}-1\}}\prod_{i'\in u}(-1)^{\xi_{i'+1}+\xi_1}\prod_{i''\in \{1,\ldots,a_{\alpha}-1\}\setminus u}1 \\
    & = (-1)^{\alpha\xi_1}\sum_{0<a_\alpha <\cdots <a_1}\prod_{i=1}^{\alpha}2^{-2a_i}(-1)^{\xi_{a_i+1}} \prod_{i'=1}^{a_{\alpha}-1}\left[ 1+(-1)^{\xi_{i'+1}+\xi_1}\right] .
  \end{align*}
The innermost product equals $2^{a_{\alpha}-1}$ if and only if $\xi_{i'+1}+\xi_1\equiv 0 \pmod 2$ for all $1\le i'< a_{\alpha}$, and equals 0 otherwise. We now focus on the condition $\xi_{i'+1}+\xi_1\equiv 0 \pmod 2$ for all $1\le i'< a_{\alpha}$. This condition is equivalent to the condition $\xi_{i'+1}=\xi_1$ for all $1\le i'< a_{\alpha}$. Thus $x$ has to be given in the form
  \begin{align*}
    (0.\underbrace{00\ldots 0}_{a_{\alpha}}***\ldots)_2\; \; \mathrm{or}\; \; (0.\underbrace{11\ldots 1}_{a_{\alpha}}***\ldots)_2 ,
  \end{align*}
where every $*$ can take either $0$ or $1$ arbitrarily. We first consider the case $\xi_1=0$. Since we have $2x=(0.\xi_2 \xi_3\ldots )_2$, the condition $\xi_2=\cdots =\xi_{a_{\alpha}}=0$ is satisfied if and only if $2x<2^{-a_{\alpha}+1}$. Next we consider the case $\xi_1=1$. Since we have
  \begin{align*}
    2-2x=(0.\underbrace{0\ldots 0}_{a_{\alpha}-1}***\ldots)_2 ,
  \end{align*}
when the condition $\xi_2=\cdots =\xi_{a_{\alpha}}=1$ is satisfied, we arrive at an equivalent condition, that is, $2-2x<2^{-a_{\alpha}+1}$. Combining these two cases, the condition we consider is satisfied if and only if $\phi(x)<2^{-a_{\alpha}+1}$. Hence the second term on the right-hand side of (\ref{eq:omega}) can be further rewritten as
  \begin{align*}
    (-1)^{\alpha\xi_1}\sum_{0<a_\alpha <\cdots <a_1}2^{a_{\alpha}-1}[ \phi(x)<2^{-a_{\alpha}+1}]  \prod_{i=1}^{\alpha}2^{-2a_i}(-1)^{\xi_{a_i+1}}  .
  \end{align*}
Substituting this result into (\ref{eq:omega}), we have thus far
  \begin{align}\label{eq:omega2}
    \omega_{\alpha}(x) = & \sum_{v=1}^{\alpha-1}(-1)^{v\xi_1}\sum_{0<a_v<\cdots <a_1}\prod_{i=1}^{v}2^{-2a_i}(-1)^{\xi_{a_i+1}} \nonumber \\
    & +  (-1)^{\alpha\xi_1}\sum_{0<a_\alpha <\cdots <a_1}2^{a_{\alpha}-1}[ \phi(x)<2^{-a_{\alpha}+1}]  \prod_{i=1}^{\alpha}2^{-2a_i}(-1)^{\xi_{a_i+1}} .
  \end{align}

For the first term on the right-hand side of (\ref{eq:omega2}), we have 
\begin{align*}
    & \quad \sum_{0<a_v<\cdots <a_1}\prod_{i=1}^{v}2^{-2a_i}(-1)^{\xi_{a_i+1}} \\
    & = \sum_{0<a_v<\cdots <a_1<m'}\prod_{i=1}^{v}2^{-2a_i}(-1)^{\xi_{a_i+1}}+ \sum_{0<a_v<\cdots <a_2<m'\le a_1}\prod_{i=1}^{v}2^{-2a_i}(-1)^{\xi_{a_i+1}} \\
    & \quad + \cdots + \sum_{0<a_v<m'\le a_{v-1}< \cdots <a_1}\prod_{i=1}^{v}2^{-2a_i}(-1)^{\xi_{a_i+1}}+ \sum_{m'\le a_v<\cdots <a_1}\prod_{i=1}^{v}2^{-2a_i}(-1)^{\xi_{a_i+1}} \\
    & = V_v(x)+\sum_{m'\le a_v<\cdots <a_1}\prod_{i=1}^{v}2^{-2a_i}(-1)^{\xi_{a_i+1}} \\
    & \quad + \sum_{t=1}^{v-1}\sum_{0<a_v<\cdots <a_{t+1}<m'}\prod_{i=t+1}^{v}2^{-2a_i}(-1)^{\xi_{a_i+1}}\sum_{m'\le a_t<\cdots <a_1}\prod_{j=1}^{t}2^{-2a_j}(-1)^{\xi_{a_j+1}} \\
    & = V_v(x)U_0+\sum_{t=1}^{v}V_{v-t}(x)\sum_{m'\le a_t<\cdots <a_1}\prod_{j=1}^{t}2^{-2a_j}(-1)^{\xi_{a_j+1}} ,
  \end{align*}
where we note that $V_0(x)=1$ for any $x\in [0,1)$. We now recall that $x\in [0,1)$ is given in the form $l2^{-m'}$ for $m'>1$, $m' \in \nat$ and $0\le l<2^{m'}$ and thus we have $\xi_{m'+1}=\xi_{m'+2}=\cdots =0$. In the last expression we have
  \begin{align}\label{eq:U_v}
    \sum_{m'\le a_t<\cdots <a_1}\prod_{j=1}^{t}2^{-2a_j}(-1)^{\xi_{a_j+1}} & = \sum_{m'\le a_t<\cdots <a_1}\prod_{j=1}^{t}2^{-2a_j} \nonumber \\
    & = \sum_{a_t=m'}^{\infty}2^{-2a_t}\sum_{a_{t-1}=a_t+1}^{\infty}2^{-2a_{t-1}}\cdots\sum_{a_1=a_2+1}^{\infty}2^{-2a_1} \nonumber \\
    & = \frac{1}{2^{2t(m'-1)}}\prod_{j=1}^{t}\frac{1}{2^{2j}-1} = U_t .
  \end{align}
Using these results and swapping the order of sums, the first term on the right-hand side of (\ref{eq:omega2}) becomes 
\begin{align*}
    \sum_{v=1}^{\alpha-1}(-1)^{v\xi_1}\sum_{0<a_v<\cdots <a_1}\prod_{i=1}^{v}2^{-2a_i}(-1)^{\xi_{a_i+1}} & = \sum_{v=1}^{\alpha-1}(-1)^{v\xi_1}\sum_{t=0}^{v}V_{t}(x)U_{v-t} \\
    & = \sum_{v=1}^{\alpha-1}(-1)^{v\xi_1}\sum_{t=1}^{v}V_{t}(x)U_{v-t}+\sum_{v=1}^{\alpha-1}(-1)^{v\xi_1}U_{v} \\
    & = \sum_{t=1}^{\alpha-1}\left(\sum_{v=t}^{\alpha-1}(-1)^{v\xi_1}U_{v-t}\right)V_t(x)+\sum_{v=1}^{\alpha-1}(-1)^{v\xi_1}U_{v} \\
    & = \sum_{t=1}^{\alpha-1}\tilde{U}_t(\xi_1)V_t(x)+(\tilde{U}_0(\xi_1)-1) ,
  \end{align*}
where the third equality is obtained by swapping the order of sums of the first term.

For the second term on the right-hand side of (\ref{eq:omega2}), we have
\begin{align}\label{eq:omega3}
    & \quad \sum_{0<a_\alpha <\cdots <a_1}2^{a_{\alpha}-1}[ \phi(x)<2^{-a_{\alpha}+1}]  \prod_{i=1}^{\alpha}2^{-2a_i}(-1)^{\xi_{a_i+1}} \nonumber \\
    & = \sum_{0<a_\alpha <\cdots <a_1<m'}2^{a_{\alpha}-1}[ \phi(x)<2^{-a_{\alpha}+1}] \prod_{i=1}^{\alpha}2^{-2a_i}(-1)^{\xi_{a_i+1}} \nonumber \\
    & \quad + \sum_{0<a_\alpha <\cdots <a_2<m'\le a_1}2^{a_{\alpha}-1}[ \phi(x)<2^{-a_{\alpha}+1}] \prod_{i=1}^{\alpha}2^{-2a_i}(-1)^{\xi_{a_i+1}} \nonumber \\
    & \quad + \cdots + \sum_{m'\le a_\alpha <\cdots <a_1}2^{a_{\alpha}-1}[ \phi(x)<2^{-a_{\alpha}+1}] \prod_{i=1}^{\alpha}2^{-2a_i}(-1)^{\xi_{a_i+1}} \nonumber \\
    & = \tilde{V}_{\alpha}(x)+\sum_{m'\le a_\alpha <\cdots <a_1}2^{a_{\alpha}-1}[ \phi(x)<2^{-a_{\alpha}+1}] \prod_{i=1}^{\alpha}2^{-2a_i}(-1)^{\xi_{a_i+1}} \nonumber \\
    & \quad + \sum_{t=1}^{\alpha-1}\sum_{0<a_\alpha<\cdots <a_{t+1}<m'}2^{a_{\alpha}-1}[ \phi(x)<2^{-a_{\alpha}+1}] \prod_{i=t+1}^{\alpha}2^{-2a_i}(-1)^{\xi_{a_i+1}} \nonumber \\
    & \quad \times \sum_{m'\le a_t<\cdots <a_1}\prod_{j=1}^{t}2^{-2a_j}(-1)^{\xi_{a_j+1}} \nonumber \\
    & = \tilde{V}_{\alpha}(x)U_0+ \sum_{t=1}^{\alpha-1}\tilde{V}_{\alpha-t}(x)U_t+\sum_{m'\le a_\alpha <\cdots <a_1}2^{a_{\alpha}-1}[ \phi(x)<2^{-a_{\alpha}+1}] \prod_{i=1}^{\alpha}2^{-2a_i}(-1)^{\xi_{a_i+1}} ,
  \end{align}
where we use the result (\ref{eq:U_v}) in the last equality. Herein we recall again that $x\in [0,1)$ is given in the form $l2^{-m'}$ for $m'>1$, $m' \in \nat$ and $0\le l<2^{m'}$. Since we also have $x\ne 0$ now, $\phi(x)$ must be greater than or equal to $2^{-a_{\alpha}+1}$ whenever $a_{\alpha}\ge m'$. Thus, we have
  \begin{align*}
    \sum_{m'\le a_\alpha <\cdots <a_1}2^{a_{\alpha}-1}[ \phi(x)<2^{-a_{\alpha}+1}]  \prod_{i=1}^{\alpha}2^{-2a_i}(-1)^{\xi_{a_i+1}} = 0,
  \end{align*}
which implies that the last term of (\ref{eq:omega3}) vanishes. Now the second term on the right-hand side of (\ref{eq:omega2}) becomes
  \begin{align*}
    & \quad (-1)^{\alpha\xi_1}\sum_{0<a_\alpha <\cdots <a_1}2^{a_{\alpha}-1}[ \phi(x)<2^{-a_{\alpha}+1}]  \prod_{i=1}^{\alpha}2^{-2a_i}(-1)^{\xi_{a_i+1}} \\
    & = (-1)^{\alpha\xi_1}\sum_{t=0}^{\alpha-1}\tilde{V}_{\alpha-t}(x)U_t = (-1)^{\alpha\xi_1}\sum_{t=1}^{\alpha}U_{\alpha-t}\tilde{V}_{t}(x) .
  \end{align*}
Therefore, we have
  \begin{align*}
    \omega_{\alpha}(x) = \sum_{t=1}^{\alpha-1}\tilde{U}_t(\xi_1)V_t(x)+(\tilde{U}_0(\xi_1)-1)+ (-1)^{\alpha\xi_1}\sum_{t=1}^{\alpha}U_{\alpha-t}\tilde{V}_t(x) ,
  \end{align*}
which completes the proof.
\end{proof}

\subsection{Fast component-by-component construction}

Here we only deal with the case of product weights, that is, $\gamma_u=\prod_{j\in u} \gamma_j$ for all $u\subseteq \{1,\ldots,s\}$, for simplicity of exposition, and show how to obtain the fast CBC construction using the fast Fourier transform. From Definition \ref{def:ho_polynomial_lattice} and Corollary \ref{corollary:quality_criterion}, we can rewrite the quality criterion $B_{\alpha,\bsgamma}(\bsq,p)$ as
  \begin{align*}
    B_{\alpha,\bsgamma}(\bsq,p) = -1+\frac{1}{2^m}\sum_{n=0}^{2^m-1}\prod_{j=1}^{s}\left[ 1+\gamma_j D_{\alpha}\omega_{\alpha}\left( v_{m'}\left( \frac{n(x)q_j(x)}{p(x)} \right)\right)\right] .
  \end{align*}
In the following, we denote
  \begin{align*}
    P_{\tau-1}(n) := \prod_{j=1}^{\tau-1}\left[ 1+\gamma_j D_{\alpha}\omega_{\alpha}\left( v_{m'}\left( \frac{n(x)q_j(x)}{p(x)} \right)\right)\right] ,
  \end{align*}
and $\bsP_{\tau-1}:=(P_{\tau-1}(1),\ldots,P_{\tau-1}(2^m-1))^\top$ for $1\le \tau \le s$, where the empty product equals 1, that is, $\bsP_{0}:=(1,\ldots,1)^\top$. Furthermore, we define a $(2^{m'}-1)\times (2^m-1)$ matrix $\bsW:=(w_{q,n})_{q\in G_{m'}\setminus \{0\},0<n<2^m}$ whose elements are given by
  \begin{align*}
    w_{q,n} := \omega_{\alpha}\left( v_{m'}\left( \frac{n(x)q(x)}{p(x)} \right)\right) .
  \end{align*}
Using these notations, what we compute in the CBC construction can be expressed as
  \begin{align}\label{eq:fast_cbc_1}
    & \quad B_{\alpha,\bsgamma}((\bsq_{\tau-1},\tilde{q}_{\tau}),p) \nonumber \\
    & = -1+\frac{1}{2^m}\sum_{n=0}^{2^m-1}\left[ 1+\gamma_{\tau} D_{\alpha}\omega_{\alpha}\left( v_{m'}\left( \frac{n(x)\tilde{q}_{\tau}(x)}{p(x)} \right)\right)\right] P_{\tau-1}(n) \nonumber \\
    & = -1+\frac{1}{2^m}\sum_{n=0}^{2^m-1}P_{\tau-1}(n)+\frac{\gamma_{\tau} D_{\alpha}}{2^m}\sum_{n=0}^{2^m-1}\omega_{\alpha}\left( v_{m'}\left( \frac{n(x)\tilde{q}_{\tau}(x)}{p(x)} \right)\right)P_{\tau-1}(n) \nonumber \\
    & = B_{\alpha,\bsgamma}(\bsq_{\tau-1},p) + \frac{\gamma_{\tau} D_{\alpha}}{2^m}\left[ \omega_{\alpha}\left( 0\right)P_{\tau-1}(0)+\sum_{n=1}^{2^m-1}w_{\tilde{q}_{\tau},n}P_{\tau-1}(n)\right] ,
  \end{align}
for $\tilde{q}_{\tau}\in G_{m'}\setminus \{0\}$. When $\tilde{q}_{\tau}=0$, we have
  \begin{align*}
    B_{\alpha,\bsgamma}((\bsq_{\tau-1},0),p) = B_{\alpha,\bsgamma}(\bsq_{\tau-1},p) + \frac{\gamma_{\tau} D_{\alpha}\omega_{\alpha}\left( 0\right)}{2^m}\sum_{n=0}^{2^m-1}P_{\tau-1}(n) ,
  \end{align*}
which can be computed at a negligibly low computational cost, so that we only consider how to compute $B_{\alpha,\bsgamma}((\bsq_{\tau-1},\tilde{q}_{\tau}),p)$ for $\tilde{q}_{\tau}\in G_{m'}\setminus \{0\}$ at a low computational cost. It is obvious from the expression (\ref{eq:fast_cbc_1}) that the choice of $\tilde{q}_{\tau}\in G_{m'}\setminus \{0\}$ affects only the term
  \begin{align}\label{eq:fast_cbc}
    \sum_{n=1}^{2^m-1}w_{\tilde{q}_{\tau},n}P_{\tau-1}(n) .
  \end{align}
Thus, we focus on this term in the following.

We consider a straightforward implementation first. In order to find $q_{\tau}$ which minimizes (\ref{eq:fast_cbc}) as a function of $\tilde{q}_{\tau}\in G_{m'}\setminus \{0\}$, one can compute a matrix-vector multiplication $\bsW\bsP_{\tau-1}$ to obtain the values of (\ref{eq:fast_cbc}) for all $\tilde{q}_{\tau}\in G_{m'}\setminus \{0\}$. This multiplication requires $O(2^{m+m'})$ arithmetic operations.

We shall consider a more elaborate implementation below. According to Algorithm \ref{algorithm:cbc}, we choose an irreducible polynomial $p\in \FF_2[x]$ with $\deg(p)=m'$. Thus, there exists a primitive element $g\in G_{m'}\setminus \{0\}$, which satisfies
  \begin{align*}
    \{g^0 \bmod{p}, g^1 \bmod{p},\ldots,g^{2^{m'}-2} \bmod{p}\}=G_{m'}\setminus \{0\} ,
  \end{align*}
and $g^{-1} \bmod{p}=g^{2^{m'}-2} \bmod{p}$. Using this property of $g$, we can obtain the values of (\ref{eq:fast_cbc}) for all $\tilde{q}_{\tau}\in G_{m'}\setminus \{0\}$ by a matrix-vector multiplication $\bsW_{\mathrm{perm}}\bsP_{\tau-1}$, where $\bsW_{\mathrm{perm}}$ is a $(2^{m'}-1)\times (2^m-1)$ matrix given by permuting the rows of $\bsW$ as
  \begin{align*}
    \bsW_{\mathrm{perm}} :=(w_{g^i \bmod{p},n})_{0\le i\le 2^{m'}-2,0<n<2^m} .
  \end{align*}
The multiplication $\bsW_{\mathrm{perm}}\bsP_{\tau-1}$, however, still requires $O(2^{m+m'})$ arithmetic operations, which can be reduced as follows.

As a first step, we add more columns to $\bsW_{\mathrm{perm}}$ to obtain a $(2^{m'}-1)\times (2^{m'}-1)$ matrix $\bsW_{\mathrm{perm}}'$ given by
  \begin{align*}
    \bsW_{\mathrm{perm}}' :=(w_{g^i \bmod{p},n})_{0\le i\le 2^{m'}-2,0<n<2^{m'}} .
  \end{align*}
We also add more elements to $\bsP_{\tau-1}$ to obtain a vector $\bsP_{\tau-1}'=(\bsP_{\tau-1}^\top,\bszero^\top)^\top$, where $\bszero$ is a vector consisting of $2^{m'}-2^m$ zeros, such that we have $\bsW_{\mathrm{perm}}\bsP_{\tau-1}=\bsW_{\mathrm{perm}}'\bsP_{\tau-1}'$.
Next we permute the rows of $\bsW_{\mathrm{perm}}'$ to obtain a $(2^{m'}-1)\times (2^{m'}-1)$ {\em circulant} matrix
  \begin{align*}
    \bsW_{\mathrm{circ}} := & (w_{g^i \bmod{p},g^{-n} \bmod{p}})_{0\le i,n\le 2^{m'}-2} \\
    = & \left( \omega_{\alpha}\left( v_{m'}\left( \frac{(g^{i-n}\bmod{p})(x)}{p(x)} \right)\right)\right)_{0\le i,n\le 2^{m'}-2}  .
  \end{align*}
Here the argument $g^{-n} \bmod{p}$ is understood as an integer by identifying the polynomial $g^{-n} \bmod{p}\in \FF_2[x]$ with an integer based on its dyadic expansion, so that the notation $w_{g^i \bmod{p},g^{-n} \bmod{p}}$ makes sense. Correspondingly, we introduce a vector $\bsQ_{\tau-1}=(Q_{\tau-1}(0),\ldots,Q_{\tau-1}(2^{m'}-2))^\top$ such that
  \begin{align*}
    Q_{\tau-1}(n) = \left\{ \begin{array}{ll}
    P_{\tau-1}(g^{-n}\bmod{p}) & \text{if} \ \deg(g^{-n}\bmod{p})<m ,  \\
    0 & \text{otherwise} ,  \\
    \end{array} \right.
  \end{align*}
for $0\le n\le 2^{m'}-2$, where the argument of $P_{\tau-1}$ is again understood as an integer as above. Using these notations, we have $\bsW_{\mathrm{perm}}'\bsP_{\tau-1}'=\bsW_{\mathrm{circ}}\bsQ_{\tau-1}$, which implies that a matrix-vector multiplication $\bsW_{\mathrm{circ}}\bsQ_{\tau-1}$ gives the values of (\ref{eq:fast_cbc}) for all $\tilde{q}_{\tau}\in G_{m'}\setminus \{0\}$.

Since the matrix $\bsW_{\mathrm{circ}}$ is circulant, the multiplication $\bsW_{\mathrm{circ}}\bsQ_{\tau-1}$ can be done efficiently by using the fast Fourier transform, requiring only $O(m'2^{m'})$ arithmetic operations \cite{NC06b}. We also refer to \cite[Section~10.3]{DP10} for details on a matrix-vector multiplication for a circulant matrix. In this way we can reduce the computational cost from $O(2^{m+m'})$ to $O(m'2^{m'})$ operations for finding $q_{\tau}$ which minimizes $B_{\alpha,\bsgamma}((\bsq_{\tau-1},\tilde{q}_{\tau}),p)$ as a function of $\tilde{q}_{\tau}\in G_{m'}$.

Suppose $q_{\tau}=g^{i^*}\bmod{p}$ minimizes $B_{\alpha,\bsgamma}((\bsq_{\tau-1},\tilde{q}_{\tau}),p)$. We update $\bsQ_{\tau}$ by
  \begin{align*}
    Q_{\tau}(n) = \left\{ \begin{array}{ll}
    Q_{\tau-1}(n)\left[ 1+\gamma_{\tau} D_{\alpha}w_{g^{i^*} \bmod{p},g^{-n} \bmod{p}}\right] & \text{if} \ \deg(g^{-n}\bmod{p})<m ,  \\
    0 & \text{otherwise} ,  \\
    \end{array} \right.
  \end{align*}
for $0<n<2^{m'}$. We then move on to the next component. In summary, the fast CBC construction proceeds as follows.

    \begin{enumerate}
        \item Choose an irreducible polynomial $p\in \FF_2[x]$ with $\deg(p)=m'$.
        \item Evaluate $\bsW_{\mathrm{circ}}$ and $\bsQ_{0}$.
        \item For $\tau=1,\ldots, s$, find $q_{\tau}$ which minimizes $B_{\alpha,\bsgamma}((\bsq_{\tau-1},\tilde{q}_{\tau}),p)$ as a function of $\tilde{q}_{\tau}\in G_{m'}$ by computing $\bsW_{\mathrm{circ}}\bsQ_{\tau-1}$ and then update $\bsQ_{\tau}$.
    \end{enumerate}

In the process 2, since the matrix $\bsW_{\mathrm{circ}}$ is circulant, we only need to evaluate one column (or one row) of $\bsW_{\mathrm{circ}}$. One column consists of $2^{m'}-1$ elements, each of which can be computed in at most $O(\alpha m')$ operations as in Theorem \ref{theorem:omega_calculation}. Thus the process 2 can be done in at most $O(\alpha m'2^{m'})$ operations. As shown in this subsection, the matrix-vector multiplication $\bsW_{\mathrm{circ}}\bsQ_{\tau-1}$ can be done in $O(m'2^{m'})$ operations and $\bsQ_{\tau}$ can be updated in $O(2^{m'})$ operations, so that the process 3 can be done in $O(sm'2^{m'})$ operations. In total, when $s$ is large as compared with $\alpha$, the fast CBC construction can be done in $O(sm'2^{m'})$ operations. As for the required memory, we are required to store one column of $\bsW_{\mathrm{circ}}$ and $\bsQ_{\tau}$, both of which need $O(2^{m'})$ memory spaces.

We now recall that we can construct good higher order polynomial lattice rules which achieve the optimal rate of the mean square worst-case error when $m'\ge \alpha m/2$, see Remark \ref{remark:cbc}. When $\alpha m$ is even, we can set $m'=\alpha m/2$. Otherwise when $\alpha m$ is odd, we have to set $m'=(\alpha m+1)/2$. Regardless of whether $\alpha m$ is even or odd, the computational cost of the fast CBC construction becomes $O(s\alpha m2^{\alpha m/2})=O(s\alpha N^{\alpha /2}\log N)$ operations using $O(2^{\alpha m/2})=O(N^{\alpha/2})$ memory. Although our obtained computational cost is much smaller than the computational cost shown in \cite{BDLNP12}, it still grows exponentially with $\alpha$. Hence, higher order polynomial lattice rules do not compare favorably with interlaced polynomial lattice rules \cite{Goxxa,Goxxb} in terms of computational cost so far, for which the computational cost grows only linearly with $\alpha$. Whether we can get rid of this exponential dependence on $\alpha$ for constructing good higher order polynomial lattice rules is open for further research.

\section*{Acknowledgments}
The author would like to thank two anonymous referees for their many helpful comments, which improves the presentation of this paper. This work was supported by Grant-in-Aid for JSPS Fellows No.24-4020.


\end{document}